\documentclass[11pt,a4paper]{article}
\usepackage[utf8]{inputenc}
\usepackage{amsmath}
\usepackage{amsfonts}
\usepackage{amssymb}
\usepackage{amsthm}
\usepackage{makeidx}
\usepackage{graphicx}
\usepackage[outdir=./]{epstopdf}

\usepackage[linesnumbered]{algorithm2e}
\usepackage{matlab-prettifier}

\theoremstyle{definition}
\newtheorem{definition}{Definition}

\theoremstyle{plain}
\newtheorem{theorem}{Theorem}

\theoremstyle{plain}
\newtheorem{lemma}{Lemma}

\theoremstyle{plain}

\theoremstyle{plain}
\newtheorem{corollary}{Corollary}

\theoremstyle{plain}

\usepackage[english]{babel}
\usepackage{csquotes}
\usepackage{authblk}
\usepackage{tikz,pgfplots}
\usepackage[backend=bibtex,style=numeric-comp, giveninits=true,maxbibnames=99,url=false, isbn=false, doi=false]{biblatex}
\bibliography{SVMbib.bib}	
\title{{Random multi-block ADMM: an ALM based view for the QP case }}
\author[1]{S. Cipolla \thanks{\texttt{scipolla@ed.ac.uk}}}
\author[1]{J. Gondzio \thanks{\texttt{j.gondzio@ed.ac.uk}}}
\affil[1]{ \footnotesize{The University of Edinburgh, School of Mathematics}}
\makeatletter
\def\keywords{\xdef\@thefnmark{}}
\makeatother

\begin{document}

\maketitle

\begin{abstract}
	Embedding randomization procedures in the Alternating Direction Method of Multipliers (ADMM) has recently attracted an increasing amount of interest as a remedy to the fact that the direct multi-block generalization of ADMM is not necessarily convergent. Even if, in practice, the introduction of such techniques could \textit{mitigate} the diverging behaviour of the multi-block extension of ADMM, from the theoretical point of view, it can ensure just the \textit{convergence in expectation}, which may not be a good indicator of its robustness and efficiency.  
	In this work, analysing the  strongly convex quadratic programming case,
	we interpret the block Gauss-Seidel sweep performed by the multi-block ADMM in the context of the inexact  Augmented Lagrangian Method. Using the proposed analysis, we are able to outline an alternative technique to those present in literature which, supported from stronger theoretical guarantees, is able to ensure the convergence of the multi-block generalization of the ADMM method. 
\end{abstract}	

{ \footnotesize
\noindent \keywords{\textbf{Keywords}: Alternating Direction Method of Multipliers, inexact  Augmented Lagrangian Method,  Randomly Shuffled Gauss-Seidel } \\
\keywords{\textbf{MSC2010 Subject Classification:} 90C25 - 65K05 - 65F10}%
}

\section{Introduction}
In this work we consider the solution of the problem:

\begin{align} \label{eq:QP_problem_ADMM}
\begin{aligned}
\min_{\mathbf{x} \in \mathbb{R}^d} \;& f(\mathbf{x}) :=\frac{1}{2}\mathbf{x}^TH\mathbf{x} +\mathbf{g}^T\mathbf{x}  \\ 
\hbox{s.t.} \; & A\mathbf{x} =\mathbf{b},
\end{aligned}
\end{align}

\noindent where $H \in \mathbb{R}^{d \times d}$ is Symmetric Positive Definite (SPD in short) and $A \in \mathbb{R}^{m \times d}$  ($d \geq m$) has full rank.

Recently, problem \eqref{eq:QP_problem_ADMM} has been widely used as a \textit{sample problem} for the convergence analysis of the $n$-block generalization of the Alternating Direction Method of Multipliers (ADMM) \cite{peng2012rasl,MR3439797,MR2765489,boyd2011distributed,eckstein2012augmented}. In particular, in \cite{MR3439797}, a counterexample in the form of problem \eqref{eq:QP_problem_ADMM} has been given to show that
the direct $n$-block extension of ADMM is not necessarily convergent when solving non-separable convex minimization problems. 
This counterexample has motivated a series of very recent works, including \cite{MR3590721,MR3121502,MR3592778, MR3740519,MR2968865,MR2983116,MR3056148,MR2968856,MR3372674,MR3374652,MR4123860,MR4066996,MR3904457,mihic2019managing,MR4174639},  where the authors analyse modifications of ADMM which ensure its convergence when $n \geq 3$.
In particular, in \cite{MR4066996,MR3904457,mihic2019managing} a series of randomization procedures  has been introduced which is  able to guarantee the convergence in expectation of the $n$-block generalization of ADMM. Since then such techniques have been  proposed as a possible remedy to the fact that the deterministic direct $n$-block extension of ADMM is not necessarily convergent.

The ADMM \cite{boyd2011distributed,eckstein2012augmented} was originally proposed in \cite{MR388811} and,
% and has recently gained in popularity for a broad spectrum of
%applications ranging from the constrained quadratic programming applied to partial differential equations  \cite{MR4152672} to emerging areas as deep-learning \cite{taylor2016training}. 
in its $n$-block version, it embeds a $n$-block Gauss-Seidel (GS)  decomposition \cite{gauss1903werke,bodewig2014matrix} into each iteration of the Augmented Lagrangian Method (ALM) \cite{MR271809,powell1969method}: the primal variables, partitioned into $n$ blocks, are cyclically updated and then a dual-ascent-type step for the dual variables is performed. 

Adopting a purely linear-algebraic approach, in the particular case of problem  \eqref{eq:QP_problem_ADMM},  ALM and ADMM can be simply interpreted in terms of matrix splitting techniques (see \cite{MR0158502,MR3495481}) for the solution of the corresponding Karush-Kuhn-Tucker (KKT) linear system (see Section \ref{sec:exactALM} and Section \ref{sec:ADMM_iALM}). 

Even if in the numerical linear algebra community the study of matrix splitting techniques for the solution of linear systems arising from saddle point problems is a well established line of research (see \cite[Sec. 8]{MR2168342} for an overview), this connection seems to be only partially exploited in  the works \cite{MR4066996,MR3904457,mihic2019managing} and,  despite the fact that analogies between ADMM and GS+ALM are apparent, to the best of our knowledge, {very few works perform a precise investigation in this direction (even in the simple case when the problem is given by equation \eqref{eq:QP_problem_ADMM}).} 
 
{Indeed,  even if it is natural to view ADMM as an approximate version of the ALM, as reported in \cite{eckstein2012augmented,MR3420805}, there were no known results in quantifying this interpretation until the very recent work \cite{MR4201710}: here the authors investigate the connection of the block symmetric Gauss–Seidel method \cite[Sec. 4.1.1]{MR3495481} with the inexact proximal ALM, which represents somehow a different setting from the one investigated here.}

Broadly speaking, {this work aims to depict a precise picture  of the synergies occurring between GS and ALM} in order to give rise to ADMM and, in turn, to shed new light on the hidden machinery which controls its convergence.	

{For the reasons explained above, our}  starting point is an analysis of the ALM from an inexact point of view and specifically   tailored for problem \eqref{eq:QP_problem_ADMM}. 
Indeed, inexact ALMs (iALM) have attracted the attention of many researchers in the last years and we refer to \cite[Sec. 1.4]{iALM} for a very recent literature review.  We mention explicitly the works \cite{MR3439811, MR3959087,MR3267149,MR3406681}, where iALM is analysed for solving linearly constrained convex programming problems, a very similar framework to the one analysed here. To the best of our knowledge, our approach does not have any evident analogy to the previously mentioned papers.

On the other hand,  the connections of the ALM  with {monotone operators/splitting methods} are well understood \cite{MR3467167,eckstein1989splitting} and,  our analysis, resembles this line of research more closely:  we use, in essence, a matrix splitting of the augmented  KKT matrix of \eqref{eq:QP_problem_ADMM} to represent the ALM/iALM iterations. It is not surprising that, as a result of this line of reasoning, we are able to relate the convergence of ALM/iALM (and their rate of convergence to an  $\varepsilon$~-~accurate primal-dual solution) to the spectral radius $\rho$ of the iteration map of a fixed point problem (see equation \eqref{eq:FIXEDiALM}).

It is important to highlight, at this stage, that encompassing inexactness in the recursion generated by a monotone operator has been extensively studied, see \cite{MR410483,MR1756912,MR1713951,MR3097289,MR1853949}.

A careful checking of the literature revealed some analogies of our approach with the inexact Uzawa's method  \cite{MR0108399}. Indeed  the ALM method can be interpreted as the Uzawa’s method applied to the augmented KKT system of problem \eqref{eq:QP_problem_ADMM}
and in the context of the inexact Uzawa's method, it is empirically well documented \cite{fortin2000augmented} and theoretically well understood \cite{MR1302679,MR1451114,MR1659105,MR1766854,MR1876634}, that a fixed number of Successive Over-Relaxation (SOR) \cite{MR46149,MR2938019} steps per inner solve (typically $10$) is needed in order to reproduce the convergence rate of the exact algorithm.  

%Surprisingly enough, our analysis allows to prove the convergence of the iALM method without any particular restriction on the parameter $\beta$ in \eqref{eq:QP_problem_ADMM} (see Theorem \ref{theo:ALMconvergence}) in contrast to the results obtained in the previously mentioned Uzawa's framework.

All the inexactness criteria developed in the  previously mentioned works are characterized by a \textit{summability condition} or a \textit{relative error} condition based on the residual previously computed.  

A first important by-product of our analysis, is that we are able to prove the convergence of the iALM without imposing any summability condition on the sequence $\{\eta^{k}\}_k$ which controls the \textit{amount of inexactness} of the iALM at $k$-th iteration (see Theorem~ \ref{theo:inexact_convergence}) also in the case when the source of inexactness is modelled using a random variable  (see Lemma~\ref{lem:epsol_iALM}). A second important advantage of our approach, is that we are able to give explicit bounds for the rate of convergence of the iALM in relation to the speed characterizing the convergence to zero of the sequence $\{\eta^k\}_{k}$. 

Beyond the previously mentioned advantages of our analysis, we trace the main contribution of this work  in the production  of an explicit  link between the accuracy required  to ensure the convergence and the specific solver used to address the minimization step in the ALM, which, in the case of problem \eqref{eq:QP_problem_ADMM}, is equivalent to the solution of a SPD linear system. Using explicit error-reduction bounds for the Conjugate Gradient (CG) method \cite{MR0060307,saig}, for the SOR method \cite{MR1280549} and its Randomly Shuffled version \cite{MR3621829},  we are able to prove that the inexactness criterion  $\eta^k = R^{k+1}$ ($R<1$ suitably user-defined), can be satisfied performing a constant number of iterations (see Theorem~\ref{theo:cgNONincreasingIts} and Theorem~\ref{theo:sorNONincreasingIts}). 
Moreover, observing that the GS decomposition is a particular case of the SOR decomposition, we are able to connect the very well known convergence issues \cite{peng2012rasl,MR2765489} of the direct $n$-block extension of ADMM (and its randomized versions \cite{MR4066996,MR3904457,mihic2019managing}) to the fact that one GS sweep for iALM-step may not be sufficient to ensure enough of the accuracy in the algorithm to deliver convergence.
Finally, as an interesting result of our analysis, we are able to propose a \textit{simple} numerical strategy aiming to mitigate, if not to eliminate entirely, the convergence issues of ADMM (see Section~\ref{sec:ADMM_iALM}): this proposal, due to its solid theoretical guarantees of convergence,  could be considered as a competitive alternative to the techniques introduced to date \cite{MR4066996,MR3904457,mihic2019managing}. We provide also computational evidence of this fact.

\subsubsection{Test Problems}  In order to showcase the developed theory, in the remainder of this work, we will consider the following test problems (all the numerical results presented are obtained using \texttt{Matlab\textsuperscript{\tiny\textregistered} R2020b}):\\
 
\textbf{Problem 1}  $H$ is the Kernel Matrix associated with the radial basis function for the data-set \texttt{heart\_scale} from \cite{CC01a} ($270$ instances, $13$ features). In particular, we consider $(H)_{ij}=e^{-\frac{\|\mathbf{x}_i-\mathbf{x}_j\|}{h^2}}$ with $h=0.5$ and $\mathbf{g}$ a random vector. For the constraints, we choose $A=\mathbf{e}^T$ where $\mathbf{e}$ is the vector of all ones and $\mathbf{b}=1$.  \\

\textbf{Problem 2} Following \cite{MR3439797},  we consider $H=h I_{3 \times 3}$ with $h=0.05$ and $\mathbf{g}$ a random vector. For the constraints we consider the matrix 
\begin{equation*}
	A=\begin{bmatrix}
	1 &  1 & 1 \\ 
	1 & 1 & 2 \\
	1 & 2 & 2
	\end{bmatrix}
\end{equation*}
and $\mathbf{b}$ a random vector ($rank(A)=3$).

\section{Augmented Lagrangian and KKT}
If we consider the Augmented Lagrangian   

\begin{equation*} 
\mathcal{L}_{\beta}(\mathbf{x}, \boldsymbol{\mu})=\frac{1}{2}\mathbf{x}^TH\mathbf{x} +\mathbf{g}^T\mathbf{x}-\boldsymbol{\mu}^T(A\mathbf{x}-\mathbf{b})+\frac{\beta}{2}\|A\mathbf{x}-\mathbf{b}\|^2,
\end{equation*}
the corresponding KKT conditions are
\begin{equation*} \label{eq:KKTconditions}
\begin{split}
&\nabla_{\mathbf{x}} \mathcal{L}_{\beta}(\mathbf{x},\boldsymbol{\mu})=\mathbf{H}\mathbf{x}+\mathbf{g}-A^T\boldsymbol{\mu}+\beta A^TA\mathbf{x}-\beta A^T\mathbf{b}=0  \\
& A\mathbf{x}-\mathbf{b}=0
\end{split}.
\end{equation*}
Multiplying by $\beta$  the second KKT condition, we  obtain the system

\begin{equation}\label{eq:KKTlS}
\underbrace{\begin{bmatrix}
	H_\beta & -A^T \\
	\beta A & {0}
	\end{bmatrix}}_{=:\mathcal{A}}\begin{bmatrix}
\mathbf{x} \\ \boldsymbol{\mu}
\end{bmatrix}=\underbrace{\begin{bmatrix}
\beta A^T \mathbf{b}-\mathbf{g} \\  \beta \mathbf{b}
\end{bmatrix}}_{=: \mathbf{q}}
\end{equation}
where $H_{\beta}:=H+\beta A^TA$. Theorem~\ref{theo:invertibility_A} states the existence of a unique solution of problem \eqref{eq:KKTlS}:

\begin{theorem} \label{theo:invertibility_A}
	The matrix $\mathcal{A}$ is invertible for all $\beta > 0$.
\end{theorem}
\begin{proof}
	Observe that
	
	\begin{equation*}
		\mathcal{A}= \begin{bmatrix}
		H_{\beta} & 0 \\
		\beta A & \beta A H_{\beta}^{-1} A^T
		\end{bmatrix}
		\begin{bmatrix}
		I & - H_{\beta}^{-1}A^T\\
		0 & I
		\end{bmatrix}.
	\end{equation*}
	The non-singularity follows using the fact that $A$ is of full rank. See also \cite[Sec. 3]{MR2168342} for different factorizations of saddle point matrices.
\end{proof}

Let us define:
\begin{definition}[$\varepsilon$~-~accurate primal-dual solution]
	We say that $[\mathbf{x},\boldsymbol{\mu}]^T$ is  an $\varepsilon$~-~accurate primal-dual solution for problem \eqref{eq:QP_problem_ADMM} if
	
	\begin{equation*}
	\| H\mathbf{x} + \mathbf{g}- A^T\boldsymbol{\mu} \| \leq \varepsilon \hbox{ and } \| A\mathbf{x}- \mathbf{b} \|\leq \varepsilon.
	\end{equation*}
	Moreover, if $[\mathbf{x}, \boldsymbol{\mu}]^T$ is a random variable, we say that it is  an expected $\varepsilon$~-~accurate primal-dual solution for problem \eqref{eq:QP_problem_ADMM} if
	
	\begin{equation*}
	\mathbb{E}(\| H\mathbf{x} + \mathbf{g}- A^T\boldsymbol{\mu}) \| \leq \varepsilon \hbox{ and } \mathbb{E}(\| A\mathbf{x}- \mathbf{b} \|)\leq \varepsilon.
	\end{equation*}
	
\end{definition}

\section{The Augmented Lagrangian Method of Multipliers (ALM)} \label{sec:exactALM}
The general form of ALM is given by

\begin{equation*} 
\left\{
\begin{array}{rc}
& \mathbf{x}^{k+1}=\min_{\mathbf{x} \in  \mathbf{R}^d} \mathcal{L}_{\beta}(\mathbf{x}, \boldsymbol{\mu}^k) \\
& \boldsymbol{\mu}^{k+1}= \boldsymbol{\mu}^{k}-\beta (A\mathbf{x}^{k+1}-\mathbf{b}),
\end{array}
\right.
\end{equation*}
which, for problem \eqref{eq:QP_problem_ADMM}, reads as
\begin{equation} \label{eq:ExactALM}
\left\{
\begin{array}{rc}
& \mathbf{x}^{k+1}=H_{\beta}^{-1}(A^T \boldsymbol{\mu}^k+\beta A^T \mathbf{b}-\mathbf{g}) \\
& \boldsymbol{\mu}^{k+1}= \boldsymbol{\mu}^{k}-\beta (A\mathbf{x}^{k+1}-\mathbf{b})
\end{array}
\right. .
\end{equation}

It is important to observe that the iterates $[\mathbf{x}^{k+1},\boldsymbol{\mu}^{k+1}]^T$ produced by \eqref{eq:ExactALM} are dual feasible, i.e.,

\begin{equation*}
0=\nabla_{\mathbf{x}} \mathcal{L}_{\beta}(\mathbf{x}^{k+1},\boldsymbol{\mu}^k)= \nabla_{\mathbf{x}} f(\mathbf{x}^{k+1})-A^T\boldsymbol{\mu}^{k+1}=H\mathbf{x}^{k+1}+\mathbf{g}-A^T\boldsymbol{\mu}^{k+1}.
\end{equation*}

It is well known that ALM  {can be derived applying the Proximal Point Method to the dual 
of problem \eqref{eq:QP_problem_ADMM}, see \cite[Sec. 6.1]{MR3467167}, but in this particular case can be also recast in  an operator splitting framework (see \cite[Sec. 7]{MR3467167}, \cite{eckstein1989splitting}):} indeed, the ALM scheme can be interpreted as a fixed point iteration obtained from a splitting {decomposition} for the KKT linear algebraic system \eqref{eq:KKTlS}  (see \cite{MR0158502,MR1829662} and \cite[Sec. 8]{MR2168342}). Writing 
$$\mathcal{A}= \begin{bmatrix}
H_{\beta} & 0 \\
\beta A & I \\
\end{bmatrix}-\begin{bmatrix}
0 & A^T \\
0 & I \\
\end{bmatrix}, $$
we can write equation \eqref{eq:ExactALM} as
\begin{equation*}
\begin{bmatrix}
\mathbf{x}^{k+1} \\
\boldsymbol{\mu}^{k+1}
\end{bmatrix}=\underbrace{\begin{bmatrix}
H_{\beta}^{-1} & 0 \\
-\beta A H_{\beta}^{-1} & I
\end{bmatrix}\begin{bmatrix}
0 & A^T \\
0 & I
\end{bmatrix}}_{=: \,G_\beta}\begin{bmatrix}
\mathbf{x}^{k} \\
\boldsymbol{\mu}^{k}
\end{bmatrix}+  \underbrace{\begin{bmatrix}
H_{\beta}^{-1} & 0 \\
-\beta A H_{\beta}^{-1} & I
\end{bmatrix}}_{=: F_\beta} \underbrace{\begin{bmatrix}
\beta A^T \mathbf{b}-\mathbf{g} \\  \beta \mathbf{b}
\end{bmatrix}}_{\mathbf{q}},
\end{equation*}
i.e., as a fixed point iteration of the form
\begin{equation*}
\begin{bmatrix}
\mathbf{x}^{k+1} \\
\boldsymbol{\mu}^{k+1}
\end{bmatrix}= G_{\beta}\begin{bmatrix}
\mathbf{x}^{k} \\
\boldsymbol{\mu}^{k}
\end{bmatrix}+ F_{\beta}\mathbf{q}.
\end{equation*}

The following Theorem~\ref{theo:G_eigs} (see \cite[Sec. 2]{MR2483050} for a similar result) is the cornerstone to prove the convergence of the ALM (see equation \eqref{eq:ExactALM}) and its inexact version (see equation \eqref{eq:IExactALM}).

\begin{theorem} \label{theo:G_eigs}
The eigenvalues of $G_{\beta}$ are s.t. $\lambda \in [0,1)$  for all $\beta>0$ and, moreover, $\rho(G_{\beta})\to 0$ for $\beta \to \infty$.
\end{theorem}
\begin{proof}
	Let us observe that $(\lambda,[\mathbf{u}, \mathbf{v}]^T)$ is an eigenpair of $G_\beta$ if and only if
	\begin{align} \label{eq:eigenpair_conditions}
	\begin{aligned}
	A^T \mathbf{v}= \lambda H_{\beta} \mathbf{u} \\
	(1-\lambda) \mathbf{v}= \lambda \beta A \mathbf{u}
	\end{aligned}.
	\end{align}
The proof is structured into three parts. \\ 

\underline{Part 1:} If $\lambda$ is an eigenvalue of $G_\beta$, then $\lambda \neq 1$.\\
By contradiction suppose that $\lambda = 1$, then from \eqref{eq:eigenpair_conditions} we have the condition
\begin{equation*}
\begin{bmatrix}
H_\beta & -A^T \\
\beta A &  {0}
\end{bmatrix} \begin{bmatrix}
\mathbf{u} \\ \mathbf{v}
\end{bmatrix}= {0},
\end{equation*}
which leads to an absurd since $\mathcal{A}$ is invertible for $\beta >0$ (see Theorem~\ref{theo:invertibility_A}). 

\underline{Part 2:} If $(\lambda,[\mathbf{u}, \mathbf{v}]^T)$ is an eigenpair  of $G_{\beta}$, then $\mathbf{u} \neq {0}$.\\
By contradiction, if $\mathbf{u}= {0}$, then from the second equation in \eqref{eq:eigenpair_conditions}, we obtain 
$(1-\lambda) \mathbf{v}= {0}$ and hence an absurd using Part 1.

\underline{Part 3.} \\
If $\mathbf{v}=0$, multiplying by $\mathbf{u}^T$ the first equation in  \eqref{eq:eigenpair_conditions}, we obtain $\lambda \mathbf{u}^TH_{\beta}\mathbf{u}=0$, which leads to $\lambda=0$ since $H_{\beta}$ is SPD.\\
If $\mathbf{v} \neq 0$, from \eqref{eq:eigenpair_conditions}, we obtain 

\begin{equation} \label{eq:eigen_eq}
	\lambda (1-\lambda) \frac{\mathbf{u}^TH_{\beta} \mathbf{u}}{\mathbf{u}^T\mathbf{u}}=\lambda \beta \frac{\mathbf{u}^TA^TA\mathbf{u}}{\mathbf{u}^T\mathbf{u}}.
\end{equation}
If in equation \eqref{eq:eigen_eq}  $\frac{\mathbf{u}^TA^TA\mathbf{u}}{\mathbf{u}^T\mathbf{u}}=0$, reasoning as before and using Part 1, we obtain $\lambda=0$. Instead, if in equation \eqref{eq:eigen_eq} we have  $\frac{\mathbf{u}^TA^TA\mathbf{u}}{\mathbf{u}^T\mathbf{u}}\neq 0$, we obtain $\lambda=0$ or $\lambda=\frac{\mathbf{u}^TH\mathbf{u}}{\mathbf{u}^TH_\beta\mathbf{u}}<1$, which completes the proof observing that, in this case, $\lambda=\frac{\mathbf{u}^TH\mathbf{u}}{\mathbf{u}^TH_\beta\mathbf{u}} \to 0$ if $\beta \to \infty$ since $\mathbf{u}^TA^TA\mathbf{u}\neq 0$.
\end{proof}

\begin{lemma} \label{lem:diagonalizability}
{The matrix $G_{\beta}$ is diagonalizable.}	
\end{lemma}
\begin{proof}
{
	Let us start observing that
	\begin{equation} \label{eq:G_beta_explicit}
		G_{\beta}=\begin{bmatrix}
		 0 & H_{\beta}^{-1}A^T \\
		 0 & I-\beta A H_{\beta}^{-1} A^T
		\end{bmatrix}.
	\end{equation}
The proof is divided into two parts. \\	
\underline{Part 1:} The matrix 	$I-\beta A H_{\beta}^{-1} A^T$ is invertible. \\
To prove this fact, it is enough to prove that $\beta A H_{\beta}^{-1} A^T$ does not have unitary eigenvalues. Using Woodbury formula and defining $C:=(I+\beta AH^{-1}A^T)^{-1}$, we have
\begin{equation*}
\begin{split}
& \beta A H_{\beta}^{-1} A^T = \beta AH^{-1}A^T(I-C\beta AH^{-1}A^T) \Rightarrow\\
&\beta A H_{\beta}^{-1} A^T\mathbf{x}=\mathbf{x} \Leftrightarrow \beta AH^{-1}A^TC\mathbf{x}=\mathbf{x}. 
\end{split} 
\end{equation*}
Thesis follows observing $\beta AH^{-1}A^TC$ and $C^{\frac{1}{2}}\beta AH^{-1}A^TC^{\frac{1}{2}} $ are similar 
and that $$\lambda(C^{\frac{1}{2}}\beta AH^{-1}A^TC^{\frac{1}{2}}) \subset (0,1).$$
\\
\underline{Part 2:} The minimal polynomial of $G_{\beta}$ factorizes in distinct linear factors. \\
The proof of this fact follows observing that the minimal polynomials of the blocks on the diagonal of $G_{\beta}$ factorize in distinct linear factors  since they are diagonalizable (see \cite[Cor 3.3.10]{MR2978290}). Moreover, since the matrix $I-\beta A H_{\beta}^{-1} A^T$ is invertible, they do not have common factors and hence their product (which coincides with the lowest common multiple ($lcm$)) is the minimal polynomial of the whole matrix. Indeed, for a general block upper triangular matrix $G$ with diagonal blocks $G_{ii}$, $i=1,\dots,n$, let us denote by $m_i(x)$ the minimal polynomials of the blocks and with $m(x)$ the minimal polynomial of the whole matrix. We have $lcm(m_i(x))|m(x)$ because $m(G_{ii})=0$. Moreover, by direct computation, one can check that, defining $s(x):=\prod_{i=1}^nm_i(x)$, it holds $s(G)=0$. If the polynomials $m_i(x)$ are pairwise relatively prime, then $s(x)=lcm(m_i(x))$ and hence $s(x)=m(x)$.

The diagonalizability of $G_{\beta}$ follows observing that, if the minimal polynomial of a given matrix factorizes in distinct linear factors, then it is diagonalizable (see, once more, \cite[Cor 3.3.10]{MR2978290}).
}
\end{proof}

\begin{lemma} \label{lem:Norm_radius_bound}
	{There exists a constant $M \equiv M(G_\beta) \geq 1$ s.t. $\|G_\beta^k\| \leq M  \rho(G_{\beta})^k$.}
\end{lemma}
\begin{proof}
Using Lemma \ref{lem:diagonalizability}, since $G_{\beta}$ is diagonalizable, we have $G_{\beta}^k=X\Lambda^kX^{-1}$, and hence 
\begin{equation}
		\|G_{\beta}^k\|\leq \underbrace{\|X\|\|X^{-1}\|}_{=:M}\|\Lambda^k\| \leq M \rho(G_\beta)^k. 
		\end{equation}
\end{proof}

\begin{definition}
	In the following, $[\overline{\mathbf{x}}, \overline{\boldsymbol{\mu}}]^T$ denotes  the unique solution of linear system \eqref{eq:KKTlS} (see Theorem~\ref{theo:invertibility_A} for existence and uniqueness).  Moreover, we define, $\rho_{\beta}:=\rho(G_{\beta}):= \max_{\lambda}\{|\lambda(G_{\beta})|\}$, $\mathbf{e}^k:=\begin{bmatrix}
	\mathbf{x}^{k}-\overline{\mathbf{x}} \\ \boldsymbol{\mu}^{k}-\overline{\boldsymbol{\mu}}
	\end{bmatrix}$, $\mathbf{d}^k:=\mathcal{A}\begin{bmatrix}
	\mathbf{x}^{k} \\ \boldsymbol{\mu}^{k}
	\end{bmatrix}- \mathbf{q}$.  
\end{definition}

\begin{theorem} \label{theo:ALMconvergence}
	The ALM in \eqref{eq:ExactALM} converges for all $\beta>0$. Moreover, we have for all $k \in \mathbb{N}$,
	
	\begin{equation*}
		\|\mathbf{e}^{k}\| \leq  \|\mathbf{e}^{0}\| {M}  \rho_{\beta}^{k}
	\end{equation*}
and

\begin{equation*}
	\|\mathbf{d}^{k}\| \leq \|\mathcal{A}\|\|\mathcal{A}^{-1}\|\|\mathbf{d}^0\| {M} \rho_{\beta}^k.
\end{equation*}
	
\end{theorem}
\begin{proof}
{From direct computation, we have
\begin{equation*}
	\begin{split}
	& \mathbf{e}^{k}=G_{\beta}^k\mathbf{e}^0,\\
	&\mathbf{d}^k=\mathcal{A}G_{\beta}^k\mathcal{A}^{-1}\mathbf{d}^0,
	\end{split}
\end{equation*}
where we used $\mathcal{A}\mathbf{e}^k=\mathbf{d}^k$. Thesis follows passing to the norms and using Lemma \ref{lem:Norm_radius_bound}. 
}	
	
%\scC{Use Lemma \ref{lem:Norm_radius_bound} in \cite[Th. 1.10]{MR0158502} and in \cite[Th. 2.14]{MR3495481}}.
\end{proof} 

\begin{figure}[ht!]
	\centering
	\includegraphics[width=\textwidth]{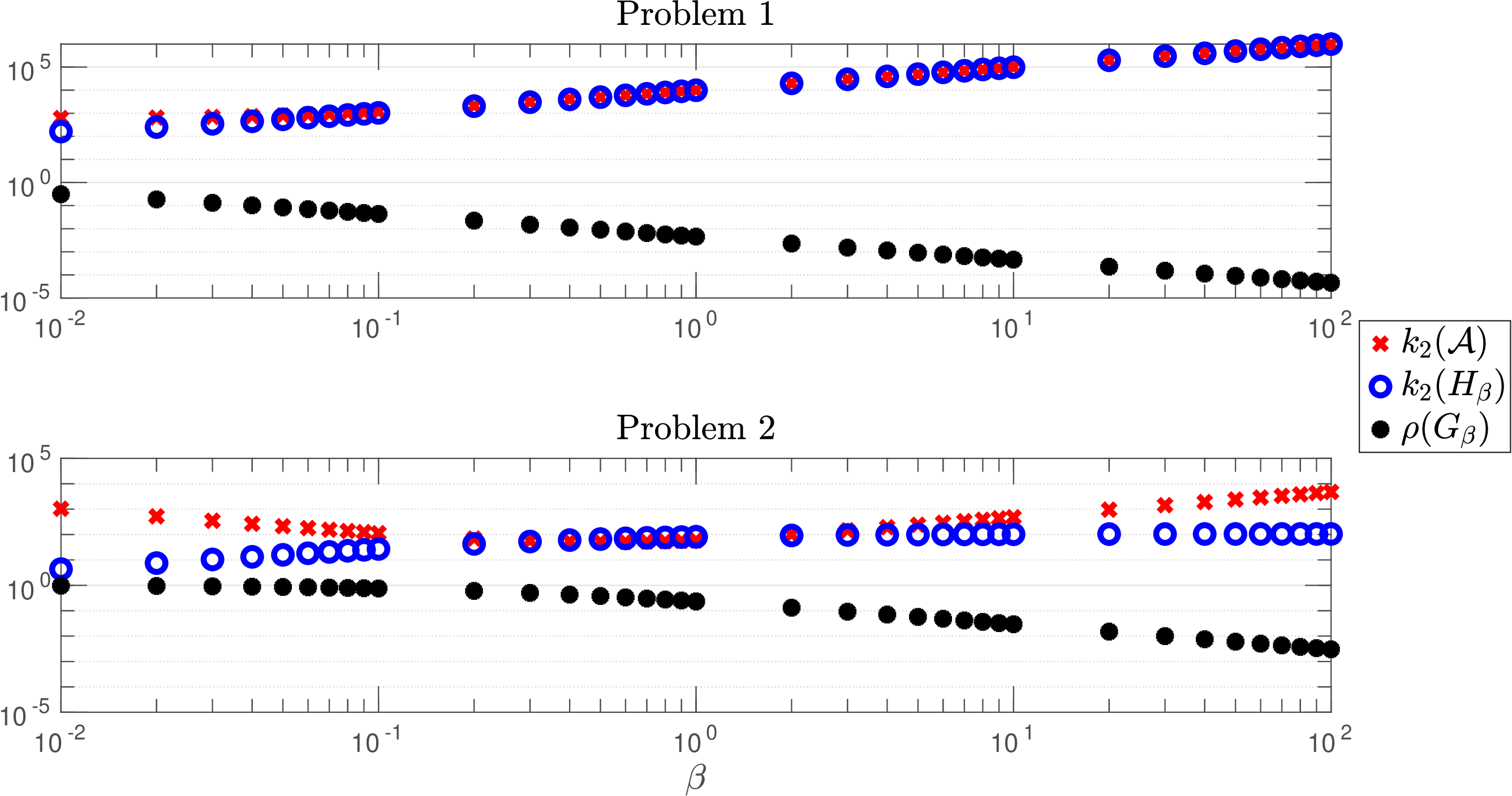}
	\caption{Behaviour of $k_2(\mathcal{A})$, $k_2(\mathcal{H}_{\beta})$, $\rho(G_{\beta})$ for different values of $\beta$ (logarithmic scale on both axes).} \label{fig:beta}
\end{figure}

In Figure~\ref{fig:beta} we report the behaviour of the condition number in $2$-norm of the matrices $\mathcal{A}$, $H_{\beta}$  (respectively $k_2(\mathcal{A})$, $k_2({H}_{\beta})$) and the spectral radius $\rho_{\beta}$ for different values of $\beta$. The results obtained in Figure~\ref{fig:beta} confirm the statement regarding $\rho_{\beta}$ in Theorem~\ref{theo:G_eigs}: the convergence of ALM can be consistently sped-up by increasing the value of $\beta$, see Theorem~\ref{theo:ALMconvergence}, but this speed-up could come at the cost of solving an increasingly ill-conditioned linear system involving $H_{\beta}$ (see the first equation in \eqref{eq:ExactALM}).
Indeed, when $\beta$ is large, the matrix $H_{\beta}$ is dominated by the term $\beta A^TA$ (see \cite[Sec. 8.1]{MR2168342} and references therein for more details) and, if $A^TA$ is singular, the condition number of the matrix $H_\beta$ progressively degrades when $\beta$ increases (see the behaviour of $k_2({H}_{\beta})$ for Problem 1 in the upper panel of Figure~\ref{fig:beta}).

The following Lemma~\ref{lem:eps_sol_compl} states the worst case complexity of ALM.

\begin{lemma} \label{lem:eps_sol_compl}
	The ALM in \eqref{eq:ExactALM} requires $O(\log_{\rho_\beta}\varepsilon)$ iterations to produce an $\varepsilon$~-~accurate primal-dual solution. 
\end{lemma}
\begin{proof}
	Observe that we have  $$\|A\mathbf{x}^k-\mathbf{b} \| \leq \frac{1}{\beta}  \|\mathbf{d}^k\|\leq  \frac{1}{\beta} \|\mathcal{A}\|\|\mathcal{A}^{-1}\|\|\mathbf{d}^0\|{M}\rho_\beta^k,$$
	where in the last inequality we used Theorem~\ref{theo:ALMconvergence}. Since, as observed at the beginning of this section, the iterates $[\mathbf{x}^k,\boldsymbol{\mu}^k]^T$ produced by the ALM are dual feasible, we have
	$\|H\mathbf{x}^k + \mathbf{g} -A^T \boldsymbol{\mu}^k\|\equiv 0$. Hence, defining $\overline{C}:= \frac{1}{\beta} \|\mathcal{A}\|\|\mathcal{A}^{-1}\|\|\mathbf{d}^0\|{M}$, we obtain that  $ k \geq \log_{\rho_\beta}(\varepsilon/\overline{C}) $ iterations of the ALM are sufficient to deliver an $\varepsilon$~-~accurate primal-dual solution.
\end{proof}

In Figure~\ref{fig:exact_convergence}, we show the behaviour of the quantities involved in the proof of Lemma~\ref{lem:eps_sol_compl} (the legend is consistent with the notation used in Lemma~\ref{lem:eps_sol_compl} {except the fact that we report $\overline{C}\equiv \overline{C}/M$}). As Lemma~\ref{lem:eps_sol_compl} states and Figure~\ref{lem:eps_sol_compl} shows, the function $\overline{C}\rho_\beta^k$ is an upper bound for the quantity $\|A\mathbf{x}^k-\mathbf{b}\|$. In this example, in order to further highlight the dependence of $\rho_{\beta}$ on $\beta$, we choose different values of $\beta$ ($\beta=0.1$ and $\beta=5$ ) such that, for Problem 1 and Problem 2, we obtain $\rho_{\beta} \approx 0.05$. Let us point out that the results reported in Figure~\ref{fig:exact_convergence} are obtained solving the  linear system in \eqref{eq:ExactALM} using a high accuracy (a direct method using \texttt{Matlab}'s ``\texttt{backslash}'' operator) and, since the iterates must be dual feasible, the residuals $\|H\mathbf{x}^k + \mathbf{g} -A^T \boldsymbol{\mu}^k\|$ are close to the machine precision.
\begin{figure}[ht!]
	\centering
	\includegraphics[width=\textwidth]{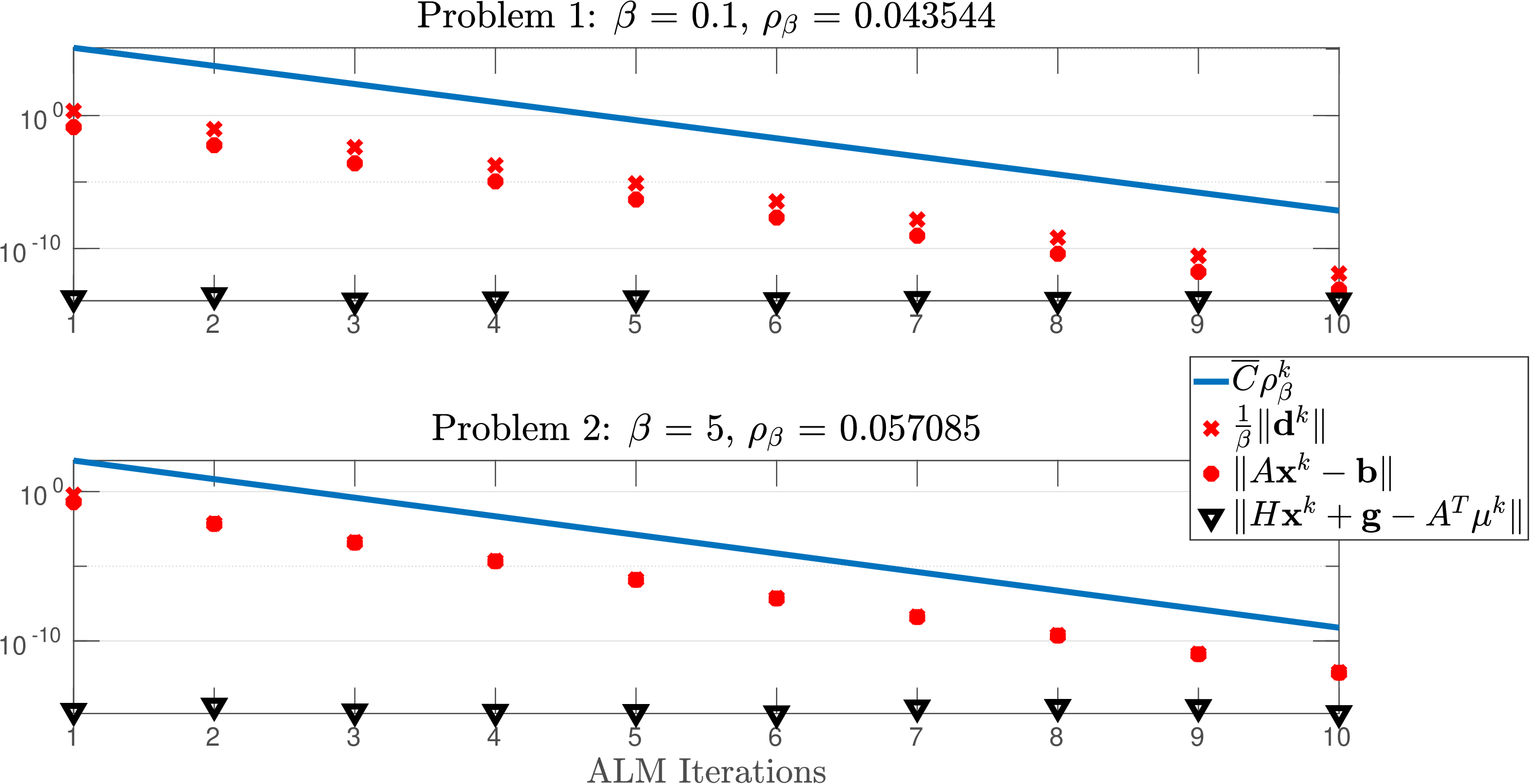}
	\caption{Behaviour of the quantities analysed in Lemma~\ref{lem:eps_sol_compl} (logarithmic scale on $y$-axis).} \label{fig:exact_convergence}
\end{figure}

\section{Inexact ALM (iALM)} \label{sec:iALM}
In this section we study in detail the iALM for problem \eqref{eq:QP_problem_ADMM}. The reader may see \cite[Sec. 1.4]{iALM} for a recent survey on this subject. In particular, we assume that the first equation in \eqref{eq:ExactALM} is not solved exactly, i.e., $\mathbf{x}^{k+1}$ is such that

\begin{equation} \label{eq:inexact_min_start}
	H_{\beta}\mathbf{x}^{k+1}-(A^T \boldsymbol{\mu}^k+\beta A^T \mathbf{b}-\mathbf{g})=\mathbf{r}^k.
\end{equation}  
In our framework, the iALM read as

\begin{equation} \label{eq:IExactALM}
\left\{
\begin{array}{rc}
& \mathbf{x}^{k+1}=H_{\beta}^{-1}(A^T \boldsymbol{\mu}^k+\beta A^T \mathbf{b}-\mathbf{g}+ \mathbf{r}^k) \\
& \boldsymbol{\mu}^{k+1}= \boldsymbol{\mu}^{k}-\beta (A\mathbf{x}^{k+1}-\mathbf{b} )
\end{array},
\right.
\end{equation}
and \eqref{eq:IExactALM} can be alternatively written as the following inexact fixed point iteration (see \cite{MR3367830} and \cite[Sec 12.2]{MR0273810} for more details on this topic): 

{\footnotesize
\begin{equation} \label{eq:FIXEDiALM}
\begin{bmatrix}
\mathbf{x}^{k+1} \\
\boldsymbol{\mu}^{k+1}
\end{bmatrix}=\underbrace{\begin{bmatrix}
	H_{\beta}^{-1} & 0 \\
	-\beta A H_{\beta}^{-1} & I
	\end{bmatrix}\begin{bmatrix}
	0 & A^T \\
	0 & I
	\end{bmatrix}}_{= \,G_\beta}\begin{bmatrix}
\mathbf{x}^{k} \\
\boldsymbol{\mu}^{k}
\end{bmatrix}+  \underbrace{\begin{bmatrix}
	H_{\beta}^{-1} & 0 \\
	-\beta A H_{\beta}^{-1} & I
	\end{bmatrix}}_{= F_\beta} \underbrace{\begin{bmatrix}
	\beta A^T \mathbf{b}-\mathbf{g}+ \mathbf{r}^k\\  \beta \mathbf{b}
	\end{bmatrix}}_{=:\mathbf{q}^k}.
\end{equation}
}
On the contrary of what was observed for the exact ALM (see the beginning of Section~\ref{sec:exactALM}), the iterates produced by \eqref{eq:IExactALM} are not dual feasible since
\begin{equation*}
0 \neq	\mathbf{r}^k=\nabla_{\mathbf{x}} \mathcal{L}_{\beta}(\mathbf{x}^{k+1},\boldsymbol{\mu}^k)=H\mathbf{x}^{k+1}+\mathbf{g}- A^T\boldsymbol{\mu}^{k+1},
\end{equation*}
i.e., the error $\mathbf{r}^k$  introduced in the solution of the first equation in \eqref{eq:IExactALM} can be interpreted as a measure of the violation of the dual feasibility condition.

In Section~\ref{sec:RSSOR&iALM} we will consider the point $\mathbf{x}^{k+1}$ in \eqref{eq:inexact_min_start} as a result of a \textit{randomized} procedure and, for this reason, we are going to present this section assuming that  $\{\mathbf{r}^k\}_k$ in \eqref{eq:FIXEDiALM} is a sequence of random variables  (and hence all the generated $\{
[\mathbf{x}^k, \boldsymbol{\mu}^k]^T\}_k$ are random variables). Moreover, all the results presented here  can be easily restated in a \textit{deterministic} framework substituting the ``\textit{almost sure (a.s.) convergence}'' with ``\textit{convergence}'' and not considering the ``\textit{expectation operator}''. For a review of the probabilistic concepts we use in the following see \cite[Ch. 2]{MR0455094}. 
%Before continuing, let us observe that the randomized framework used shares some ideas from \cite{MR3513271,loizou2019convergence,MR3179953,MR2970257,MR3432148}.

The following Theorem~\ref{theo:inexact_convergence} addresses the convergence of the iALM using the inexact fixed point formulation in \eqref{eq:FIXEDiALM}. 

\begin{theorem} \label{theo:inexact_convergence}
	
	 Let $\beta > 0$. If $\lim_{j \to \infty} \|\mathbf{r}^j\|=0$ a.s., then the iALM in \eqref{eq:IExactALM} converges a.s. to the solution of the linear system \eqref{eq:KKTlS} and the following inequalities hold a.s. for every $k \in \mathbb{N}$:
	
	\begin{equation*} 
	\|\mathbf{e}^k\|\leq {M} \rho_\beta^k\|\mathbf{e}^0\|+{M}\|F_{\beta}\|\sum_{j=0}^{k-1}\rho_\beta^{k-1-j}\|\mathbf{r}^{j}\|,
	\end{equation*}
	
	\begin{equation}\label{eq:inexact_residual}
	\|\mathbf{d}^{k}\| \leq \|\mathcal{A}\|\|\mathcal{A}\|^{-1} \|\mathbf{d}^0\|{M} \rho_\beta^k+{M}\|\mathcal{A}\|\|F_{\beta}\|\sum_{j=0}^{k-1}\rho_\beta^{k-1-j}\|\mathbf{r}^{j}\|.
	\end{equation}
	
\end{theorem}
\begin{proof}
	If $[\overline{\mathbf{x}}, \overline{\boldsymbol{\mu}}]^T$ is a solution of \eqref{eq:KKTlS}, then it satisfies the fixed point equation
	
	\begin{equation} \label{eq:FIXEPointsolution}
	\begin{bmatrix}
	\overline {\mathbf{x}} \\
	\overline {\boldsymbol{\mu}}
	\end{bmatrix}={\begin{bmatrix}
		H_{\beta}^{-1} & 0 \\
		-\beta A H_{\beta}^{-1} & I
		\end{bmatrix}\begin{bmatrix}
		0 & A^T \\
		0 & I
		\end{bmatrix}}\begin{bmatrix}
	\overline {\mathbf{x}} \\
	\overline {\boldsymbol{\mu}}
	\end{bmatrix}+  {\begin{bmatrix}
		H_{\beta}^{-1} & 0 \\
		-\beta A H_{\beta}^{-1} & I
		\end{bmatrix}} {\begin{bmatrix}
		\beta A^T \mathbf{b}-\mathbf{g} \\  \beta \mathbf{b}
		\end{bmatrix}}.
	\end{equation}
Subtracting \eqref{eq:FIXEPointsolution} from \eqref{eq:FIXEDiALM}
we obtain
\begin{equation*}
	\mathbf{e}^k=G_{\beta}\mathbf{e}^{k-1}+F_{\beta}\begin{bmatrix}
	\mathbf{r}^{k-1} \\ 0
	\end{bmatrix} \hbox{ a.s. }
\end{equation*}
and hence
\begin{equation} \label{eq:error_inexact}
{
	\mathbf{e}^k=G^k_{\beta}\mathbf{e}^{0}+\sum_{j=0}^{k-1}G_{\beta}^{k-1-j}F_{\beta}\begin{bmatrix}
	\mathbf{r}^{j} \\ 0
	\end{bmatrix} \hbox{ a.s. }.}
\end{equation}

{Passing to the norms in \eqref{eq:error_inexact} and using Lemma \ref{lem:Norm_radius_bound}}, we have
\begin{equation} \label{eq:inexact_series_convergence} 
\|\mathbf{e}^k\|\leq \rho_\beta^k {M}\|\mathbf{e}^0\|+{M}\|F_{\beta}\|\sum_{j=0}^{k-1}\rho_\beta^{k-1-j}\|\mathbf{r}^{j}\| \hbox{ a.s..} 
\end{equation}

The a.s. convergence to zero of $\{ \|\mathbf{e}^k \|\}_k$ follows from \eqref{eq:inexact_series_convergence} observing that, if $\lim_{k \to \infty} \|\mathbf{r}^k\|=0$ a.s., 
then  $$\lim_{k \to \infty}  \sum_{j=0}^{k-1}\rho_\beta^{k-1-j}\|\mathbf{r}^{j}\|=0 \hbox{ a.s. }$$ (this is a particular case of the Toeplitz Lemma, see \cite[Exercise 12.2-3]{MR0273810} for the deterministic case, \cite{MR3574725} and references therein for the probabilistic case). The second part of the statement follows observing that
\begin{equation*}
	\|\mathbf{d}^k\|=\|\mathcal{A}\mathbf{e}^k\| \leq \|\mathcal{A}\|\|\mathbf{e}^k\| \hbox{ a.s. }
\end{equation*}
and that $\|\mathbf{e}^0\|\leq \|\mathcal{A}\|^{-1} \|\mathbf{d}^0\|$.
\end{proof}

\begin{lemma} \label{lem:epsol_iALM}
 Suppose $\mathbb{E}(\|\mathbf{r}^j\|) \leq R^{j+1}$ for all $j \in \mathbb{N}$ and $R<1$. Then the iterates of the iALM in \eqref{eq:IExactALM} converge a.s. to the solution of the linear system \eqref{eq:KKTlS}. Moreover, if $R<\rho_{\beta}$, then $O(\log_{\rho_\beta}\varepsilon)$ iterations are sufficient to produce an expected $\varepsilon$~-~accurate primal-dual solution; else, if  $\rho_{\beta} \leq R$, then $O(\log_{R}\varepsilon)$ iterations are sufficient (given that $\varepsilon$ is sufficiently small).
\end{lemma}

\begin{proof}
	If $\mathbb{E}(\|\mathbf{r}^j\|) \leq R^{j+1}$ for all $j \in \mathbb{N}$,  then $\sum_{j=0}^{\infty} \mathbb{E}(\|\mathbf{r}^j\|) < \infty $ and hence, using \cite[Th. 2.1.3]{MR0455094}, we have $\lim_{j \to \infty} \|\mathbf{r}^j\|=0$ a.s.. Using now Theorem~\ref{theo:inexact_convergence}, we have that $\|\mathbf{d}^k\|$ converges a.s. to zero.
	
    Using equation  \eqref{eq:inexact_residual} and the hypothesis $\mathbb{E}(\|\mathbf{r}^j\|) \leq R^{j+1}$, we have
	
	\begin{equation} \label{eq:series_dk_upper}
	\mathbb{E}(\|\mathbf{d}^{k}\|) \leq \|\mathcal{A}\|\|\mathcal{A}\|^{-1}\|\mathbf{d}^0\|{M}\rho_\beta^k+ {M}\|\mathcal{A}\|\|F_{\beta}\|\sum_{j=0}^{k-1}\rho_\beta^{k-1-j}R^{j+1}.
	\end{equation}

	Let us observe, moreover, that 
	$$\mathbb{E}(|\|H\mathbf{x}^k + \mathbf{g} -A^T \boldsymbol{\mu}^k\| -\|\beta A^T(A\mathbf{x}^k- \mathbf{b})\| |) \leq \mathbb{E}(\|H_{\beta}\mathbf{x}^k-A^T \boldsymbol{\mu}^k +\mathbf{g} -\beta A^T \mathbf{b}  \|)\leq \mathbb{E}(\|\mathbf{d}^k\|), $$
	and hence
	
	\begin{equation} \label{eq:dual_dk_upper}
	\mathbb{E}(\|H\mathbf{x}^k + \mathbf{g} -A^T \boldsymbol{\mu}^k\|) \leq \mathbb{E}(\|\mathbf{d}^k\|)+\|A^T\|\mathbb{E}(\|\mathbf{d}^k\|) \leq C_1 \mathbb{E}(\|\mathbf{d}^k\|),
	\end{equation}
	where we defined $C_1:=(1+\|A^T\|)$ and used the fact that $\|\beta(A\mathbf{x}^k-\mathbf{b})\|\leq\|\mathbf{d}^k\|$ a.s..

	\underline{Case $R<\rho_{\beta}$.} Using \eqref{eq:series_dk_upper}, we have
	
	$$\mathbb{E}(\|\mathbf{d}^k\|) \leq \|\mathcal{A}\|\|\mathcal{A}\|^{-1} \|\mathbf{d}^0\|{M} \rho_\beta^k+ \rho_\beta^{k}{M}\|\mathcal{A}\|\|F_{\beta}\|\frac{R}{\rho_\beta}\sum_{j=0}^{k-1}(\frac{R}{\rho_\beta})^{j} \leq C_2\rho_\beta^{k},$$
	where $C_2:=\max\{{M}\|\mathcal{A}\|\|\mathcal{A}\|^{-1} \|\mathbf{d}^0\|, {M}\frac{\frac{R}{\rho_\beta}\|\mathcal{A}\|\|F_\beta\|}{1-\frac{R}{\rho_\beta}} \}$.
	
	Moreover, using  the above inequality, we have also
	
	$$\mathbb{E}(\|A\mathbf{x}^k-\mathbf{b} \|) \leq \frac{1}{\beta} \mathbb{E}(\|\mathbf{d}^k\|)\leq  \frac{1}{\beta} C_2\rho_\beta^k,$$
	and hence, using \eqref{eq:dual_dk_upper} and defining $\overline{C}:=\max\{C_1C_2,\frac{1}{\beta}C_2\}$,  we obtain that $k \geq \log_{\rho_\beta}(\varepsilon/\overline{C}) $ iterations of iALM are sufficient to produce an expected  $\varepsilon$~-~accurate primal-dual solution.
	
	\underline{Case $\rho_{\beta} \leq R$.} Using \eqref{eq:series_dk_upper}, we have
	
	$$\mathbb{E}(\|\mathbf{d}^k\|) \leq \|\mathcal{A}\|\|\mathcal{A}^{-1}\|\|\mathbf{d}^0\|{M} R^k+ R^{k}k{M}\|\mathcal{A}\|\|F_{\beta}\| \leq C_2R^{k}k,$$
	where $C_2:=\max\{{M}\|\mathcal{A}\|\|\mathcal{A}^{-1}\|\|\mathbf{d}^0\|, {M}\|\mathcal{A}\|\|F_\beta\| \}$.
	Let us observe that, in this case, we have
	$$\mathbb{E}(\|A\mathbf{x}^k-\mathbf{b} \|) \leq \frac{1}{\beta}  \mathbb{E}(\|\mathbf{d}^k\|)\leq  \frac{1}{\beta}C_2R^kk ,$$
	and hence, using \eqref{eq:dual_dk_upper} and defining $\overline{C}:=\max\{C_1C_2,\frac{1}{\beta}C_2\}$,  we obtain that to produce an expected $\varepsilon$~-~accurate primal-dual solution it suffices to perform 
	$ k+ \log_{R}k \geq \log_{R}(\varepsilon/\overline{C}) $ iterations of iALM.
	The last part of the statement follows observing that $\lim_{k \to \infty} \frac{k+ \log_{R}k}{k}=1$.
\end{proof}

Before concluding this section, let us state the following Corollary~\ref{cor:limit_finitness}, which will be used later:

\begin{corollary}\label{cor:limit_finitness}
	Suppose $\mathbb{E}(\|\mathbf{r}^j\|) \leq R^{j+1}$ for all $j \in \mathbb{N}$ and $R<1$. If $R>\rho_\beta$, then 
	\begin{equation}
	\frac{\|\beta(A\mathbf{x}^k-\mathbf{b})\|}{R^k} \leq  L < \infty \hbox{ a.s. for every } k \in \mathbb{N},
	\end{equation}
	and hence, we have
	\begin{equation} \label{eq:limit_finitness}
	\mathbb{E}( \frac{\|\beta(A\mathbf{x}^k-\mathbf{b})\|}{R^k}) \leq  L \hbox{ for every } k \in \mathbb{N}.
	\end{equation}

\end{corollary}
\begin{proof}
%	Using Lemma~\ref{lem:epsol_iALM}, since $\|\mathbf{d}^k\|$ converges to zero a.s. and since $\|\beta(A\mathbf{x}^k-\mathbf{b})\|\leq \|\mathbf{d}^k\|$, we have that $\|\beta(A\mathbf{x}^k-\mathbf{b})\|$ converges to zero a.s.. 
%	
	Using \eqref{eq:inexact_residual} we have
	\begin{equation*} 
	\begin{split}
	& \frac{\|\beta(A\mathbf{x}^k-\mathbf{b})\|}{R^k}\leq \frac{\|\mathbf{d}^k\|}{R^k} \leq   \\
	&  {M}\|\mathcal{A}\|\|\mathcal{A}^{-1}\|\|\mathbf{d}^0\|(\frac{\rho_\beta}{R})^k+{M}\|\mathcal{A}\|\|F_{\beta}\|\sum_{j=0}^{k-1}(\frac{\rho_\beta}{R})^{k-1-j},
	\end{split}
	\end{equation*}
	from which thesis follows observing that $\sum_{j=0}^{k-1}(\frac{\rho_\beta}{R})^{k-1-j} \leq \frac{1}{1-\frac{\rho_\beta}{R}}$ for all $k$.
\end{proof}

\section{The solution of the linear system} \label{sec:SLS}
In this section, \textit{given} $[\mathbf{x}^k, \boldsymbol{\mu}^k]$, we suppose that the linear system
\begin{equation} \label{eq:inexat_min}
H_{\beta}\mathbf{x}=(A^T \boldsymbol{\mu}^k+\beta A^T \mathbf{b}-\mathbf{g}),
\end{equation}
is solved using an iterative solver; in particular, we will consider two different methods for the solution of the SPD system in \eqref{eq:inexat_min}, namely the Conjugate Gradient (CG) method \cite{MR0060307} and a Randomly Shuffled version of the Successive Over-Relaxation (RSSOR) method \cite{MR3621829}. 

%It is interesting to note that the \textit{increasing accuracy condition} for the expected residual in Lemma~\ref{lem:epsol_iALM},  i.e., $\mathbb{E}(\|\mathbf{r}^k\|) \leq R^{k+1}$, for the solution of the linear system \eqref{eq:inexat_min} seems to suggest that the number of iterations of the chosen iterative solver should increase when the iterates of iALM proceeds.  In this section we will show that this is not the case if $R > \rho_{\beta}$.

Since $\mathbf{r}^k$ in the first equation of \eqref{eq:IExactALM} is the (possibly randomized) residual  associated to the linear system \eqref{eq:inexat_min},  i.e.,
\begin{equation*} 
H_{\beta}\mathbf{x}^{k+1}-(A^T \boldsymbol{\mu}^k+\beta A^T \mathbf{b}-\mathbf{g})=\mathbf{r}^k,
\end{equation*}    
one would be tempted to think that the \textit{increasing accuracy condition} for the expected residual in Lemma~\ref{lem:epsol_iALM},  i.e., $\mathbb{E}(\|\mathbf{r}^k\|) \leq R^{k+1}$, requires that the number of iterations of the chosen iterative solver increases when the iterates of iALM proceed.  In this section we will show that this is not the case if $R > \rho_{\beta}$.

    For the remaining of this work let us define
 \begin{equation*}
 	\boldsymbol{\chi}^k:=A^T \boldsymbol{\mu}^k+\beta A^T \mathbf{b}-\mathbf{g},
 \end{equation*}
 and $\{\eta^k\}_k\to 0$ as the \textit{forcing sequence} such that $\mathbb{E}(\|\mathbf{r}^k\|) \leq \eta^k$ for all $k \in \mathbb{N}$.
 
 We use, moreover, the following inequalities: given $B \in \mathbb{R}^{d \times d}$ SPD, if we order the eigenvalues of $B$ as $\lambda_1(B)\geq \dots \lambda_d(B)$, it holds
  \begin{equation}\label{eq:SPDmatrix norm}
  	\lambda_{d}(B)\|\mathbf{x}\|^2_B \leq \|B\mathbf{x}\|^2 \leq \lambda_1(B)\|\mathbf{x}\|^2_B \hbox{ for all } \mathbf{x}\in \mathbb{R}^{d}
  \end{equation}
  and
  \begin{equation}\label{eq:SPDmatrix norm2}
  	\lambda_{d}(B)\|\mathbf{x}\|^2\leq \|B^{1/2}\mathbf{x}\|^2 \leq \lambda_1(B)\|\mathbf{x}\|^2 \hbox{ for all } \mathbf{x}\in \mathbb{R}^{d}.
  \end{equation}

 \subsection{Conjugate Gradient Method} \label{sec:CG}
 In this subsection we suppose that the linear system \eqref{eq:inexat_min} is solved using the Conjugate Gradient (CG) method and hence, all the results presented in Section \ref{sec:iALM} will be used in the \textit{deterministic} case. The following Theorem~\ref{theo:cgrate} addresses the rate of convergence of CG:
 
 \begin{theorem}(\cite[Th. 6.29]{saig}) \label{theo:cgrate}
 	Consider the linear system $B\mathbf{y}=\boldsymbol{\chi}$ where $B$ is a SPD matrix and $\overline{\mathbf{y}}$ is its solution. Then the iterates $\{\mathbf{y}^j\}_j$  produced by the CG method satisfy 
 	\begin{equation*} \label{eq:cgRate}
 		\|\overline{\mathbf{y}}- \mathbf{y}^j\|_B \leq 2\big [ \frac{\sqrt{{k}_2(B)}-1}{\sqrt{k_2(B)}+1} \big]^j\|\overline{\mathbf{y}}- \mathbf{y}^0\|_B.
 	\end{equation*}
 \end{theorem}
 
 Applying Theorem~\ref{theo:cgrate} to the solution of the linear system \eqref{eq:inexat_min} and setting  $\mathbf{y}^0=\mathbf{x}^k$ for every $k \in \mathbb{N}$, we have,
 
 \begin{equation} \label{eq:CGforiALM}
 \|\overline{\mathbf{x}}^{k+1}- \mathbf{x}^{k+1, j}\|_{H_{\beta}} \leq 2\big [ \frac{\sqrt{{k}_2(H_{\beta})}-1}{\sqrt{k_2(H_{\beta})}+1} \big]^j\|\overline{\mathbf{x}}^{k+1}- \mathbf{x}^{k}\|_{H_{\beta}},
 \end{equation}
 where $H_{\beta}\overline{\mathbf{x}}^{k+1}=\boldsymbol{\chi}^k$ and $\{\mathbf{x}^{k+1, j}\}_j$ is the sequence generated by CG to approximate $\overline{\mathbf{x}}^{k+1}$.
 
 \begin{theorem} \label{theo:cgNONincreasingIts}
 	Let $\{\eta^k\}_k=R^{k+1}$ with $R>\rho_\beta$. Define 
 	\begin{equation}\label{eq:cg_min_guarantee}
 		\overline{j}^k:= \min\{j \;:\; \|\mathbf{r}^{k,j}\| \leq \eta^k \}
 	\end{equation}
 	where $\{\mathbf{r}^{k,j}:=H_\beta \mathbf{x}^{k+1,j}-\boldsymbol{\chi}^k\}_j$ is the sequence of residuals generated  from CG.
 	Then, there exists $ \overline{j} \in \mathbb{N}$ such that 
 	$\overline{j} \geq \overline{j}^k$ for all $k$. Moreover, an $\varepsilon$~-~accurate primal-dual solution of problem \eqref{eq:QP_problem_ADMM} can be obtained in $O(\log_{R}\varepsilon)$ iALM iterations.
 \end{theorem}
 
 \begin{proof}
 	Using \eqref{eq:SPDmatrix norm} in \eqref{eq:CGforiALM} we have
 	
 	\begin{equation*}
 		\|H_{\beta}\mathbf{x}^{k+1,j}-\boldsymbol{\chi}^k\|\leq 2\big [ \frac{\sqrt{{k}_2(H_{\beta})}-1}{\sqrt{k_2(H_{\beta})}+1} \big]^j \sqrt {k_2(H_{\beta})} \|H_{\beta}\mathbf{x}^{k}-\boldsymbol{\chi}^k\|
 	\end{equation*}
 	and hence, if
 	
 	\begin{equation}\label{eq:CGiterEstimate}
 		j^k := \lceil \log_{\frac{\sqrt{{k}_2(H_{\beta})}-1} {\sqrt{k(H_{\beta})}+1}} \frac{\eta^k}{2\sqrt{k_2(H_{\beta})}\|H_{\beta}\mathbf{x}^{k}-\boldsymbol{\chi}^k\|} \rceil,
 	\end{equation}
 	 then   $\|\mathbf{r}^{k,j^k}\| \leq \eta^k$. Observe, moreover, that using the second equation in \eqref{eq:IExactALM} for the expression of $\boldsymbol{\chi}^k$, we have
 	 
 	 \begin{equation}\label{eq:prev_res_newrhs}
 	 \|H_{\beta}\mathbf{x}^{k}-\boldsymbol{\chi}^k\|=\|\mathbf{r}^{k-1}+\beta A^T(A\mathbf{x}^k-\mathbf{b})\|\leq \|\mathbf{r}^{k-1}\|+\|A^T\| \|\beta(A\mathbf{x}^k-\mathbf{b})\|.
 	 \end{equation}
 	 % 	  and that $\|\beta(A\mathbf{x}^k-\mathbf{b})\|\leq	\|\mathbf{d}^k\|\leq C\rho_\beta^k$
%    for some constant $C>0$ independent from $k$ (see Lemma~\ref{lem:eps_sol_compl} for the first inequality and Lemma~\ref{lem:epsol_iALM} for the second one), we have
% 	 
% 	where we used the hypothesis $\|\mathbf{r}^{k-1}\|\leq \eta^k\leq \mu \rho_\beta^{k}$.  
Using equation \eqref{eq:prev_res_newrhs}, we have
 	 
 	 \begin{equation} \label{eq:existimations_from_below}
 	 \begin{split}
 	 \frac{\eta^k}{2\sqrt{k_2(H_{\beta})}\|H_{\beta}\mathbf{x}^{k}-\boldsymbol{\chi}^k\|} \geq & \frac{\eta^k}{2\sqrt{k_2(H_{\beta})}(\|\mathbf{r}^{k-1}\|+\|A^T\| \|\beta(A\mathbf{x}^k-\mathbf{b})\|)} \geq\\
 	  & \frac{R^{k+1}}{2\sqrt{k_2(H_{\beta})}R^k(1+\frac{\|A^T\| \|\beta(A\mathbf{x}^k-\mathbf{b})\|}{R^k})}
 	 \end{split}.
 	 \end{equation}
 	 
 	 Using now equation \eqref{eq:limit_finitness} (deterministic case) in equation \eqref{eq:existimations_from_below} we can state the existence of a constant $C >0 $ such that, for all $k \in \mathbb{N}$, we have
 	 
 	 \begin{equation*}
 	 	\frac{R}{2\sqrt{k_2(H_{\beta})}(1+\frac{\|A^T\| \|\beta(A\mathbf{x}^k-\mathbf{b})\|}{R^k})} \geq C,
 	 \end{equation*}
 	and hence 
 	\begin{equation*}
 	\overline{j}:= \lceil \log_{\frac{\sqrt{{k}_2(H_{\beta})}-1} {\sqrt{k(H_{\beta})}+1}} C \rceil \geq j^k \hbox{ for all } k.
 	\end{equation*}
 The first part of the statement follows observing that  $j^k\geq \overline{j}^k$ for all $k$.  
 The last part of the statement follows, instead, observing that the hypotheses of Lemma~\ref{lem:epsol_iALM} are satisfied.
 	\end{proof}
\begin{corollary} \label{cor:mtrix_iALM_CG}
	If problem \eqref{eq:QP_problem_ADMM} is solved using the iALM  in \eqref{eq:IExactALM} and each sub-problem is solved using CG as in Theorem~\ref{theo:cgNONincreasingIts}, then $O(\log_{R} \varepsilon)$ matrix-vector multiplications involving $H_{\beta}$
	are sufficient to obtain an $\varepsilon$~-~accurate primal-dual solution.
\end{corollary}
\begin{proof}
 It follows	from Theorem~\ref{theo:cgNONincreasingIts} and Lemma~\ref{lem:epsol_iALM} observing that each step of CG requires one matrix-vector product involving $H_{\beta}$. 
\end{proof}

In the upper panels of Figure~\ref{fig:inexact_CG} we report the quantities analysed in the proof of Lemma~\ref{lem:epsol_iALM} (the legend is consistent with the notation used in Lemma~\ref{lem:epsol_iALM} { except for the fact that we report $\overline{C} \equiv \overline{C}/M$ }). As reported in the proof Lemma~\ref{lem:epsol_iALM}  and confirmed by Figure~\ref{fig:inexact_CG}, when $R>\rho_\beta$, the function $\overline{C}R^kk$ is an upper bound for the quantities $\|A\mathbf{x}^k-\mathbf{b}\|$ and  $\|H\mathbf{x}^k + \mathbf{g} -A^T \boldsymbol{\mu}^k\|$. In the lower panels of Figure~\ref{fig:inexact_CG} we report the quantity $\overline{j}^k$ in equation \eqref{eq:cg_min_guarantee} obtained using CG for iALM step. In this example, in order to further highlight the fact that the number of CG iterations does not increase when the iALM iterations proceed (see Theorem~\ref{theo:cgNONincreasingIts}), we slowed down the speed of convergence of the iALM increasing $\rho_{\beta}$ and then choosing $R=\rho_{\beta}+10^{-2}$ (see also the numerical results reported in Figure~\ref{fig:exact_convergence} to have a term of comparison). 

Concerning the results obtained for Problem 2, it is interesting to note that the very fast decay of the dual residuals $\|H\mathbf{x}^k + \mathbf{g} -A^T \boldsymbol{\mu}^k\|$ is due to the fact that, in this case, 
CG can be considered as \textit{a direct method} since $H_{\beta}$ is reasonably well conditioned (see Figure~\ref{fig:beta}) and of small dimension. This is, indeed, a very similar behaviour of that observed  in Figure~\ref{fig:exact_convergence}, where the linear systems involving $H_{\beta}$  are solved using a direct method. 

\begin{figure}[ht!]
	\centering
	\includegraphics[width=\textwidth]{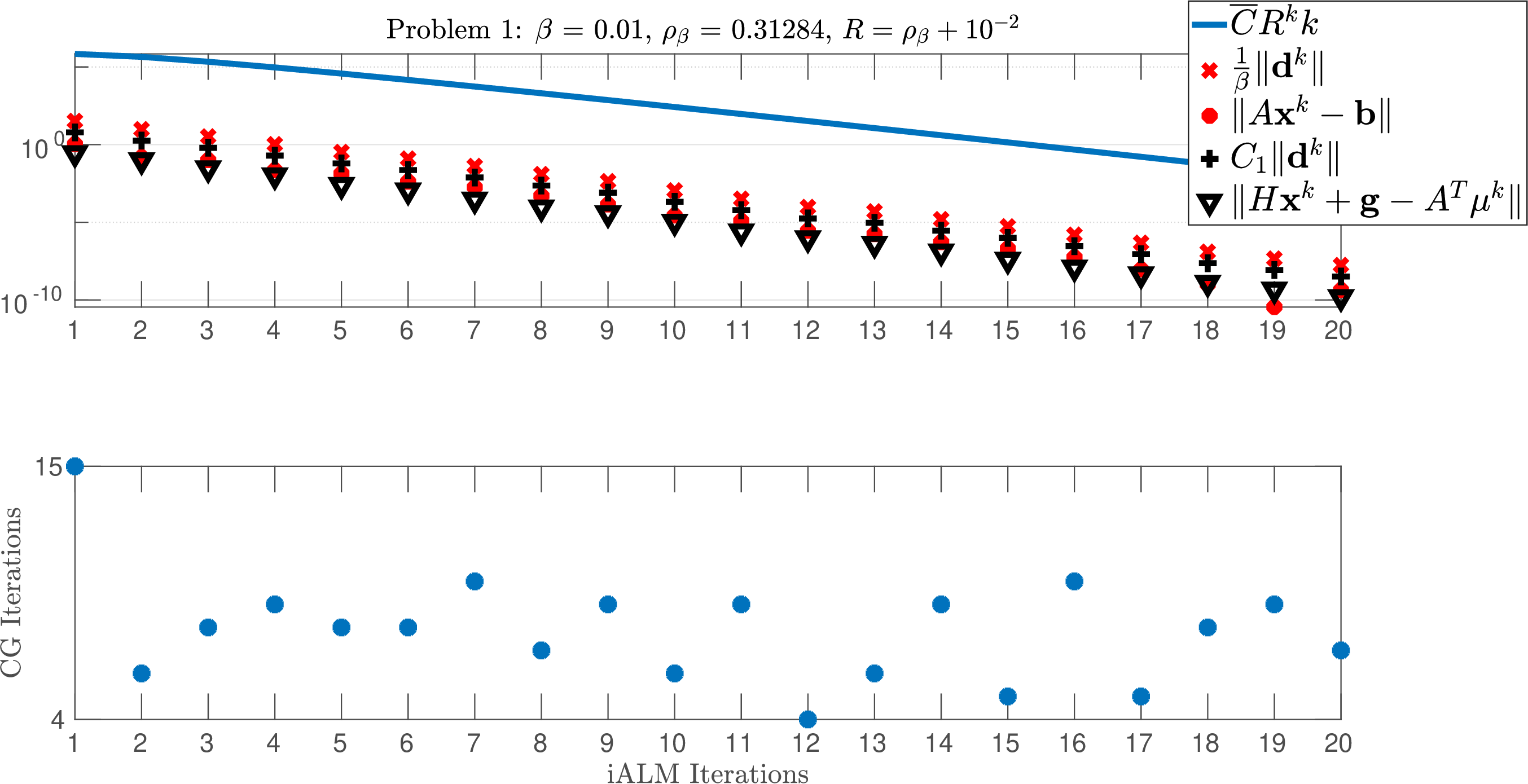}\\
	\includegraphics[width=\textwidth]{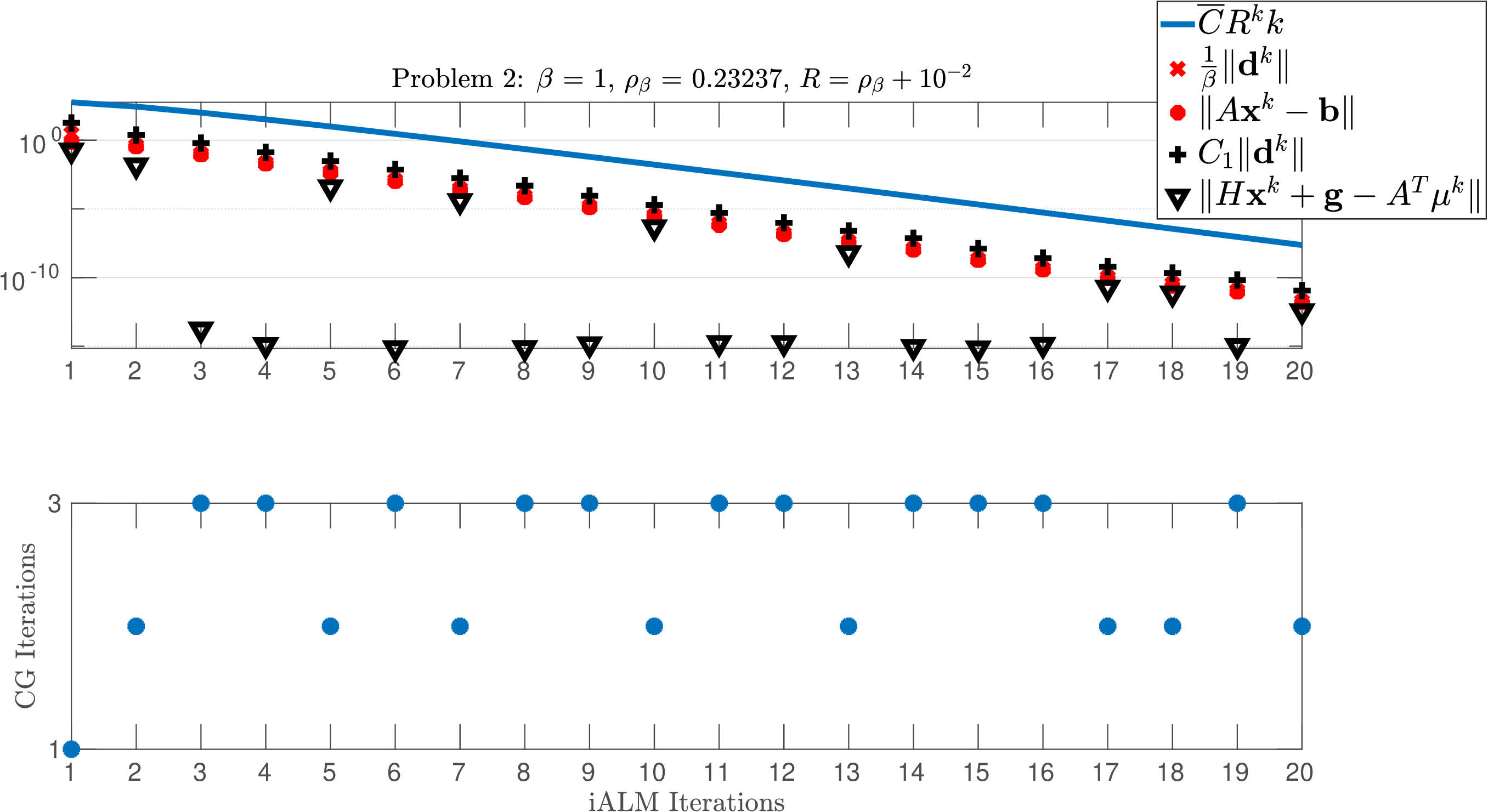}
	\caption{Upper panels: behaviour of the quantities analysed in Lemma~\ref{lem:epsol_iALM} (logarithmic scale on $y$-axis). Lower Panels: $\overline{j}^k$ in equation \eqref{eq:cg_min_guarantee} when CG is employed for the solution of \eqref{eq:inexat_min} using $\{\eta^k\}_k$ and $\{\mathbf{x}^{k+1, 0}\}_k$ as in Theorem~\ref{theo:cgNONincreasingIts}. } \label{fig:inexact_CG}
\end{figure} 

\subsection{SOR and Randomly Shuffled SOR}
 In this subsection we suppose that the linear system in equation \eqref{eq:inexat_min} is solved using the block Successive Over-Relaxation method (SOR) \cite{MR46149,MR2938019} or its Randomly Shuffled version (RSSOR) \cite{MR3621829}. Adapting the existing error-reduction results for the SOR method to our purposes, requires a slightly greater effort than in the CG case (see Section~\ref{sec:CG}). For this reason and for the sake of completeness, before presenting our results, we deliver a  brief survey on the block SOR method which  is based on \cite{MR0158502,MR3495481,MR3621829}.
 \subsubsection{A brief survey on SOR \cite{MR0158502}}
 Let $B \in \mathbb{C}^{d \times d}$. Consider the linear system 
 \begin{equation} \label{eq:linear_system_SOR}
 B\mathbf{y}= \boldsymbol{\chi}.
 \end{equation}
 We can express the matrix $B$ as the sum of block-matrices $B=D-L-U$ where 
 
 {\footnotesize
 	\begin{equation} \label{eq:SOR_SPLIT}
 	\begin{split}
 	&D:=\begin{bmatrix}
 	B_{1,1} & &  & \\
 	& B_{2,2} &   &\\ 
 	&  & \ddots  & \\
 	&  &    & B_{n,n}\\
 	\end{bmatrix}, \;	
 	L:=-\begin{bmatrix}
 	0_{1,1} &0 & 0 &0 \\
 	B_{2,1}& 0_{2,2} & 0  &\vdots \\ 
 	\vdots & \ddots & \ddots  & 0\\
 	B_{n,1}& B_{n,2} & \dots  & 0_{n,n}  
 	\end{bmatrix},\\
 	& U:=-\begin{bmatrix}
 	0_{1,1} &B_{1,2} &  & B_{1,n} \\
 	0 & 0_{2,2} & \ddots  &  \vdots \\ 
 	\vdots &  & \ddots  & B_{{n-1},n} \\
 	0 & 0& \dots  & 0_{n,n}  
 	\end{bmatrix}.
 	\end{split}	
 	\end{equation}
 }
 Let us suppose now that the block-diagonal matrix $D$ is invertible. The fixed point problem corresponding to equation \eqref{eq:linear_system_SOR} can be written as
 
 \begin{equation*}
 (D-\omega L)\mathbf{y}=((1-\omega) D+\omega U)\mathbf{y}+\omega \boldsymbol{\chi}
 \end{equation*}
 and the SOR method is defined as
 
 \begin{equation} \label{eq:SOR_NN}
 \mathbf{y}^{j+1}=(D-\omega L)^{-1}((1-\omega) D+\omega U)\mathbf{y}^j+\omega(D-\omega L)^{-1} \boldsymbol{\chi}.
 \end{equation}
 The Gauss-Seidel (GS) method is recovered for $\omega=1$. Observe that equation \eqref{eq:SOR_NN} can be written alternatively as
 \begin{equation} \label{eq:SOR_NN2}
 \mathbf{y}^{j+1}=(I-\omega D^{-1}L)^{-1}((1-\omega) I+\omega D^{-1}U)\mathbf{y}^j+\omega(I- \omega D^{-1}L)^{-1}D^{-1}\boldsymbol{\chi},
 \end{equation}
 and for this reason, usually, the \textit{point successive over-relaxation matrix} is defined  as
 
 \begin{equation*}
 \mathcal{L}_{\omega}:=(I-\omega D^{-1}L)^{-1}((1-\omega)I+\omega D^{-1} U).
 \end{equation*}
 
 The following Corollary of the Ostrowski-Reich Theorem states the convergence of the block SOR iteration: 
 \begin{corollary}(\cite[Cor. 3.14]{MR0158502})
 	Let $B \in \mathbb{C}^{n \times n}$ and $D,L,U$ be defined as in \eqref{eq:SOR_SPLIT}. If $D$ is positive definite, then the {block SOR} method in \eqref{eq:SOR_NN} is convergent for all $\mathbf{y}^0$ if and only if $0<\omega<2$ and $B$ is positive definite. 
 \end{corollary}

In this work we are going to deal just with symmetric matrices and, for this reason, we denote the factor $U$ in \eqref{eq:SOR_SPLIT} with $L^T$.  It is worth noting, moreover,  that using the equality $(1-\omega)D+\omega L^T=(D-\omega L)- \omega B$, we can further rewrite the SOR iteration in \eqref{eq:SOR_NN} as
\begin{equation} \label{eq:SOR_energy_norm_conv}
\mathbf{y}^{j+1}=(I-\omega (D-\omega L)^{-1}B)\mathbf{y}^j+\omega (D- \omega L)^{-1}\boldsymbol{\chi}.
\end{equation}

In \cite{MR3621829}, a Randomly Shuffled version of SOR (RSSOR) has been introduced and studied: it is obtained considering $P^{j}$ as a random permutation matrix (with uniform distribution and independent from the current guess $\mathbf{y}^j$) and applying the SOR splitting to the linear system $P^jB{P^j}^TP^j\mathbf{x}=P^j\boldsymbol{\chi}$, i.e., considering $$P^jB{P^j}^T=D_{P^j}-L_{P^j}-L_{P^j}^T.$$ 
The RSSOR is defined as 

\begin{equation} \label{eq:shuflled_sor}
\mathbf{y}^{j+1}=(I-\omega {P^j}^T(D_{P^j}-\omega L_{P^j})^{-1}P^jB)\mathbf{y}^{j}+\omega {P^j}^T(D_{P^j}-\omega L_{P^j})^{-1}P^j \boldsymbol{\chi}.
\end{equation}
Moreover, let us observe that after defining $Q_{P^j}:=\omega {P^j}^T(D_{P^j}-\omega L_{P^j})^{-1}P^j$,  \eqref{eq:shuflled_sor} can be written as a function of the random variables $P^1,\dots,P^j$, i.e., 

\begin{equation} \label{eq:shuflled_sor_compact}
\mathbf{y}^{j+1}= \prod_{\ell=0}^j (I-Q_{P^\ell}B)  \mathbf{y}^0 + \sum_{i=0}^j(\prod_{\ell=i+1}^j(I-Q_{P^{\ell}}))Q_{P^i}\boldsymbol{\chi},
\end{equation}
where we set $(\prod_{\ell=i+1}^j(I-Q_{P^{\ell}})):= I$ if $\ell+1>j$.

Before concluding this section, let us point out that the main idea connected with RSSOR is related to the fact that, although the spectral distribution of the matrix $PBP^T$ does not depend on any particular permutation matrix $P$, the spectrum of the lower triangular part $D_P-L_P$ does depend on it. As a result, also the spectral radius of the matrix $\mathcal{L}_{\omega}^P:=(I-\omega {P}^T(D_{P}-\omega L_{P})^{-1}PB)$ is affected by the particular choice of $P$. To further highlight the aforementioned dependence and to strengthen the intuition of the reader in this regard, in Figure~\ref{fig:sr_RSGS} we report  $\rho(\mathcal{L}_{1}^P)$ for all the permutation matrices $P$ and for $100$ randomly generated matrices of the form $B=R^TR+I \in \mathbb{R}^{7 \times 7}$ ($R$ is generated using the Matlab's function ``\texttt{rand}''). 

\begin{figure}[ht!]
	\centering
	\includegraphics[trim={6cm 7cm 1cm 7cm},clip,width=\textwidth]{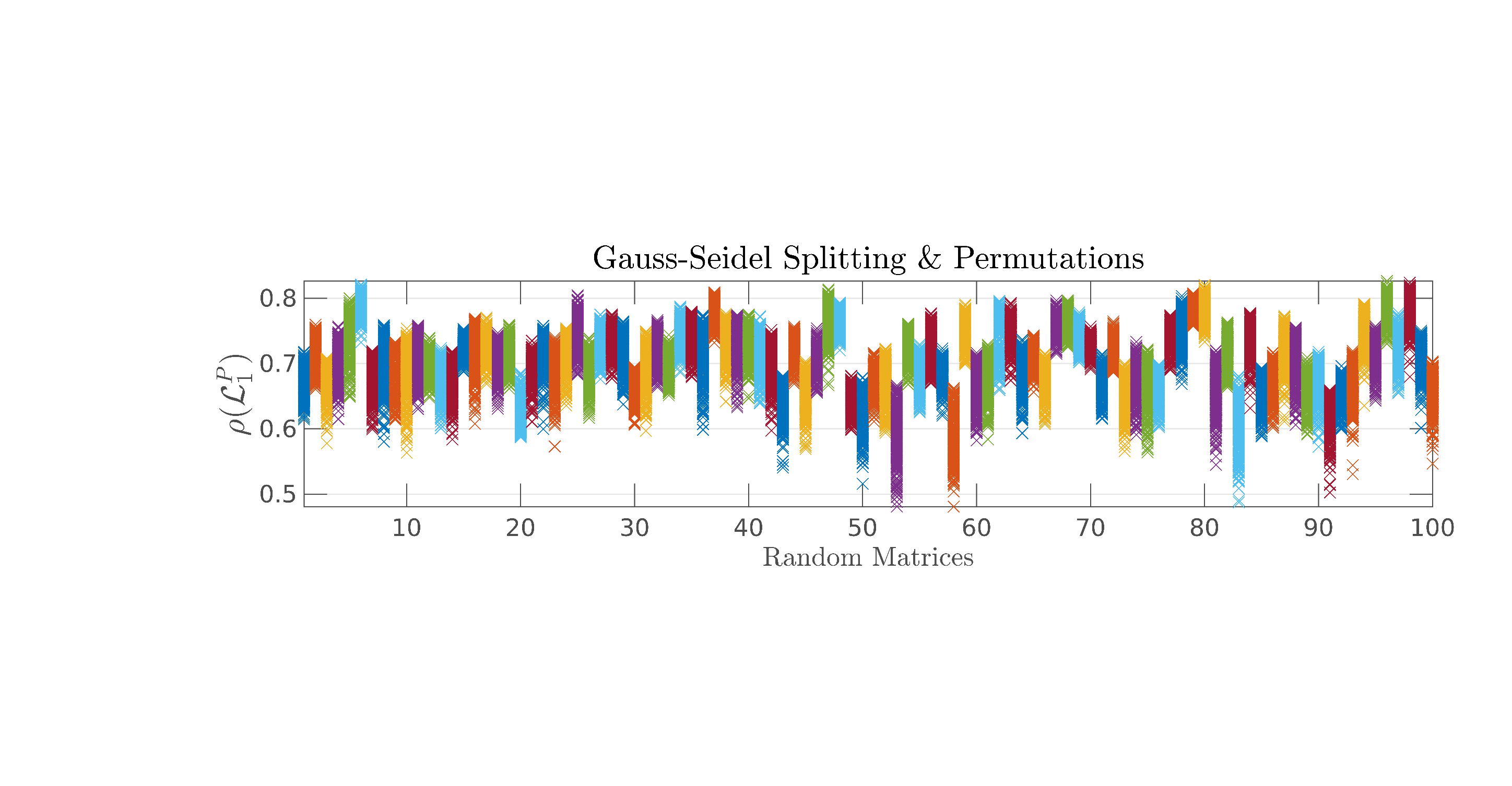}
	\caption{$\rho(\mathcal{L}_{1}^P)$ for $100$ random matrices of the form $B=R^TR+I \in \mathbb{R}^{7 \times 7}$.} \label{fig:sr_RSGS}
\end{figure}

 \subsubsection{Rate of Convergence of SOR}
 This section is based on \cite{MR3621829}. If $B$ is SPD and is partitioned as in \eqref{eq:SOR_SPLIT},
 the linear system in \eqref{eq:linear_system_SOR} can be transformed as 
 \begin{equation} \label{eq:normalized_diagonal_SOR}
 	D^{-1/2}BD^{-1/2}D^{1/2}\mathbf{y}=D^{-1/2}\boldsymbol{\chi}
 \end{equation}
($D$ is SPD since $B$ is SPD) and hence the coefficient matrix can be decomposed as
\begin{equation} \label{eq:normalized_diagonal_splitting_SOR}
	D^{-1/2}BD^{-1/2}=I-D^{-1/2}LD^{-1/2}-(D^{-1/2}LD^{-1/2})^T,
\end{equation}  
where $D^{-1/2}LD^{-1/2}$ and $(D^{-1/2}LD^{-1/2})^T$ are, respectively, strictly lower triangular and strictly upper triangular. For the above explained reasons, in this section we will suppose that $B=I-L-L^T$. 

Observe, moreover, that the SOR method applied to the system in \eqref{eq:normalized_diagonal_SOR} with the splitting \eqref{eq:normalized_diagonal_splitting_SOR} coincides exactly with \eqref{eq:SOR_NN2} and hence, the fact that in this section we suppose that the diagonal of $B$ is the identity, is expected to simplify the presentation.
 
The following Theorem~\ref{theo:SOR_NR} gives a precise bound for the rate of convergence of the SOR method:
 
 \begin{theorem}(\cite[Th. 1]{MR3621829}) \label{theo:SOR_NR}
 	Let $B$ a SPD matrix, then the SOR method \eqref{eq:SOR_energy_norm_conv} converges for $0<\omega < 2$ in the energy norm associated with $B$ according to
 	
 	\begin{equation}\label{eq:non_random_sor_rate}
 	\|\overline{\mathbf{y}}-\mathbf{y}^j\|_{B}^2\leq (1- \frac{(2-\omega)\omega \lambda_1(B)}{(1+\frac{1}{2}\lfloor \log_2(2d) \rfloor \omega \lambda_1(B))^2 {k}_2(B)})^j	\|\overline{\mathbf{y}}-\mathbf{y}^0\|_{B}^2.
 	\end{equation}
 \end{theorem}
 
 The rate of convergence stated in \eqref{eq:non_random_sor_rate} depends on the dimension of the problem $d$ and this feature is not desirable for large scale problems. 
 
 One of the main advantages of the RSSOR consists in the fact that the expected error reduction factor is independent from the dimension of the problem, as stated in the following:
 
 \begin{theorem}(\cite[Th. 4]{MR3621829}) \label{th:oswald_theorem} 
 	The expected squared energy norm error of the RSSOR iteration converges exponentially with the bound
 	\begin{equation}\label{eq:random_SOR_RC}
 	\mathbb{E}(\|\overline{\mathbf{y}}-\mathbf{y}^j\|_{B}^2)\leq (1- \frac{(2-\omega)\omega \lambda_1(B)}{(1+\omega \lambda_1(B))^2 {k}_2(B)})^j	\|\overline{\mathbf{y}}-\mathbf{y}^0\|_{B}^2.
 	\end{equation}
 	for any $\omega \in (0,2)$.
 \end{theorem}

 As already pointed out, equation \eqref{eq:random_SOR_RC} does not exhibit any dependence on the dimension of the problem and, for this reason, the Randomly Shuffled versions of SOR should be considered for large scale problems. Moreover, the following corollary addresses the convergence of the iterates to the solution of the linear system:
 
 \begin{corollary} 
 	$\lim_{j \to \infty}\|\overline{\mathbf{y}}-\mathbf{y}^j\|_{B}^2=0$ a.s..
 \end{corollary}
 \begin{proof}
 	Using \eqref{eq:random_SOR_RC},  we have that  $\sum_{j=0}^\infty \mathbb{E}(\|\overline{\mathbf{y}}-\mathbf{y}^j\|_{B}^2) < \infty$. Thesis follows applying \cite[Th. 2.1.3]{MR0455094}.
 	
% 	Using Markov inequality, for every $\delta >0$, we have $\mathbb{P}(\|\overline{\mathbf{y}}-\mathbf{y}^j\|_{B}^2 > \delta) \leq  \frac{\mathbb{E}(\|\overline{\mathbf{y}}-\mathbf{y}^j\|_{B}^2)}{\delta} $ and hence $\sum_{j=0}^\infty \mathbb{P}(\|\overline{\mathbf{y}}-\mathbf{y}^j\|_{B}^2 > \delta) < \infty$. Theis follows by Borel-Cantelli Lemma.
 \end{proof}

 \subsubsection{Using SOR in iALM} \label{sec:RSSOR&iALM}
 We are ready to analyse the behaviour of SOR method in the framework of the iALM \eqref{eq:IExactALM}. 
 In particular, we are going to present our results for the RSSOR method (see equation \eqref{eq:shuflled_sor}), but analogous techniques/results apply/hold for the non-randomized version \eqref{eq:SOR_energy_norm_conv}. This choice is mainly driven by the reasons of timeliness:
 in the next Section~\ref{sec:ADMM_iALM} we are able to interpret the recently introduced Randomized ADMM (RADMM) as a particular case of  iALM where the linear system \eqref{eq:inexat_min} is solved (inexactly) using RSSOR with $\omega=1$ (which will be denoted, in the following, as Randomly Shuffled Gauss-Seidel (RSGS)).  For this reason, in this section, we apply the results presented in Section \ref{sec:iALM} in the \textit{probabilistic} form considering $\{\mathbf{r}^k\}_k$ and $\{[\mathbf{x}^k, \boldsymbol{\mu}^k]^T\}_k$ as sequences of random variables.
 
 Of course, the same results as presented here hold, with \textit{simple} modifications, for the deterministic ADMM and the classical GS method. 
 
In order to use the rate of convergence stated in \eqref{eq:random_SOR_RC}, we write $H_{\beta}=D-L-L^T$ and transform the linear system in \eqref{eq:inexat_min} as follows:
 
 \begin{equation} \label{eq:diagonal_scaled}
 	D^{-1/2}H_{\beta}D^{-1/2}D^{1/2}\mathbf{x}=D^{-1/2}\boldsymbol{\chi}^k.
 \end{equation}
 Let us define $\widetilde{H}_{\beta}=D^{-1/2}H_{\beta}D^{-1/2}$,    $\widetilde{\boldsymbol{\chi}}^k:=D^{-1/2}\boldsymbol{\chi}^k$,  $\widetilde{\mathbf{x}}:=D^{1/2}\mathbf{x}$.

 Consider, moreover, the random variable
  $$\mathbb{E}(\|\widetilde{\overline{\mathbf{x}}}^{k+1}- \widetilde{\mathbf{x}}^{k+1, j}\|_{\widetilde{H}_{\beta}}^2| \begin{bmatrix}
 \mathbf{x}^k  \\ 
 \boldsymbol{\mu}^k 
 \end{bmatrix} ),$$
 where $\widetilde{H}_{\beta}\widetilde{\overline{\mathbf{x}}}^{k+1}=\widetilde{\boldsymbol{\chi}}^k$ and $\{\widetilde{\mathbf{x}}^{k+1, j}\}_j$ is the random sequence generated by RSSOR method in \eqref{eq:shuflled_sor}  to approximate $\widetilde{\overline{\mathbf{x}}}^{k+1}$, i.e., the solution of problem \eqref{eq:diagonal_scaled}.

The following Lemma \ref{lem:conditional_expectation} will be useful to state the main result of this section: 

\begin{lemma} \label{lem:conditional_expectation}
Let us suppose that the RSSOR in equation \eqref{eq:shuflled_sor}  is used for the solution of the linear system \eqref{eq:diagonal_scaled} with $\mathbf{y}^0=D^{1/2}\mathbf{x}^{k}=:\widetilde{\mathbf{x}}^{k+1, 0}$. 
If the random variable $(P^{k+1,0}, \dots, P^{k+1,j})$ is independent from $\{[\mathbf{x}^k, \boldsymbol{\mu}^k]^T\}_k$ for every $j,k \in \mathbb{N}$ (beyond the standard assumptions required on the $P^{k+1,j}$ by RSSOR), then
	\begin{equation}\label{eq:SORforiALM}
	\mathbb{E}(\|\widetilde{\overline{\mathbf{x}}}^{k+1}- \widetilde{\mathbf{x}}^{k+1, j}\|_{\widetilde{H}_{\beta}}) \leq (1- \frac{(2-\omega)\omega \lambda_1(\widetilde{H}_{\beta})}{(1+\omega \lambda_1(\widetilde{H}_{\beta}))^2 {k}_2(\widetilde{H}_{\beta})})^{j/2} \mathbb{E}(\|\widetilde{\overline{\mathbf{x}}}^{k+1}- \widetilde{\mathbf{x}}^{k+1, 0}\|_{\widetilde{H}_{\beta}}).
	\end{equation}
\end{lemma}
\begin{proof}
Let us observe that, using \eqref{eq:shuflled_sor_compact}, we can write $$\|\widetilde{\overline{\mathbf{x}}}^{k+1}- \widetilde{\mathbf{x}}^{k+1, j}\|_{\widetilde{H}_{\beta}}^2=g((P^{k+1,0},\dots,P^{k+1,j-1}), \begin{bmatrix}
\mathbf{x}^k  \\
\boldsymbol{\mu}^k
\end{bmatrix}),$$ where $g$ is a deterministic function.

Using the fact that, if the random variable $Y$ is independent from $X$ (see \textit{Freezing Lemma}, \cite[Example 5.1.5]{MR2722836}), it holds
$$\mathbb{E}(g(Y, X)|X)= \mathbb{E}(g(Y,x))_{|x=X},$$ and using \eqref{eq:random_SOR_RC}, we have

\begin{equation*} 
\mathbb{E}(\|\widetilde{\overline{\mathbf{x}}}^{k+1}- \widetilde{\mathbf{x}}^{k+1, j}\|_{\widetilde{H}_{\beta}}^2| \begin{bmatrix}
\mathbf{x}^k  \\ 
\boldsymbol{\mu}^k 
\end{bmatrix} ) \leq (1- \frac{(2-\omega)\omega \lambda_1(\widetilde{H}_{\beta})}{(1+\omega \lambda_1(\widetilde{H}_{\beta}))^2 {k}_2(\widetilde{H}_{\beta})})^j \|\widetilde{\overline{\mathbf{x}}}^{k+1}- \widetilde{\mathbf{x}}^{k+1, 0}\|_{\widetilde{H}_{\beta}}^2 \hbox{ a.s..}
\end{equation*}
Moreover, using the conditional Jensen's Inequality in the left hand-side of the previous equation (see \cite[Th. 34.4]{MR2893652}) and then passing the square root,  we have
\begin{equation*} 
\mathbb{E}(\|\widetilde{\overline{\mathbf{x}}}^{k+1}- \widetilde{\mathbf{x}}^{k+1, j}\|_{\widetilde{H}_{\beta}}| \begin{bmatrix}
\mathbf{x}^k  \\ 
\boldsymbol{\mu}^k 
\end{bmatrix} ) \leq (1- \frac{(2-\omega)\omega \lambda_1(\widetilde{H}_{\beta})}{(1+\omega \lambda_1(\widetilde{H}_{\beta}))^2 {k}_2(\widetilde{H}_{\beta})})^{j/2} \|\widetilde{\overline{\mathbf{x}}}^{k+1}- \widetilde{\mathbf{x}}^{k+1, 0}\|_{\widetilde{H}_{\beta}} \hbox{ a.s..}
\end{equation*}
Thesis follows considering the expectation on both sides of the above inequality and using the properties of the conditional expectation \cite[Th. 34.4]{MR2893652}.
\end{proof}

% 
% where $\widetilde{H}_{\beta}\widetilde{\overline{\mathbf{x}}}^{k+1}=\widetilde{\boldsymbol{\chi}}^k$ and $\widetilde{\mathbf{x}}^{k+1, j}$ is the random sequence generated by RSSOR to approximate $\widetilde{\overline{\mathbf{x}}}^{k+1}$, i.e., the solution of problem \eqref{eq:diagonal_scaled}.
% Applying  the RSSOR method in \eqref{eq:shuflled_sor} for the solution of the linear system \eqref{eq:diagonal_scaled} and using \eqref{eq:random_SOR_RC}, we have for every $k \in \mathbb{N}$,
We are now ready to state the following Theorem~\ref{theo:sorNONincreasingIts} which summarizes the properties of the iALM  in \eqref{eq:IExactALM} when each sub-problem is solved using RSSOR:

 \begin{theorem} \label{theo:sorNONincreasingIts}
 Let $\{\eta^k\}_k=R^{k+1}$ with $R>\rho_{\beta}$.
 	Define 
 	\begin{equation}\label{eq:min_RRSOR}
 		\overline{j}^k:= \min\{j \;:\; \mathbb{E}(\|\mathbf{r}^{k,j}\|) \leq \eta^k \},
 	\end{equation}
 where $\{\mathbf{r}^{k,j}:=H_\beta \mathbf{x}^{k+1,j}-\boldsymbol{\chi}^k\}_j$ is the sequence of random residuals generated by RSSOR.
 Then, there exists $ \overline{j} \in \mathbb{N}$ such that 
 $\overline{j} \geq \overline{j}^k$ for all $k$.
%Define $\{j_k\}_k$ as the sequence of random variables giving the number of RSSOR iterations sufficient to ensure ${\mathbb{E}(\|\mathbf{r}^{k,j_k}\|)} \leq \eta^k$.
% Then there exists a constant $k_0 \in \mathbb{N}$ and ${C}>0$ such that, 
% 	$$\mathbb{E}(j_k) \geq \log_{(1- \frac{(2-\omega)\omega \lambda_1(\widetilde{H}_{\beta})}{(1+\omega \lambda_1(\widetilde{H}_{\beta}))^2 {k}_2(\widetilde{H}_{\beta})})}{C}  $$ 
% 	for all  $k\geq k_0$. 
Moreover, an expected $\varepsilon$~-~accurate primal-dual solution of problem \eqref{eq:QP_problem_ADMM} can be obtained in $O(\log_{R}\varepsilon)$ iALM iterations.  
 \end{theorem}
 
 \begin{proof}
 Using \eqref{eq:SPDmatrix norm} in \eqref{eq:SORforiALM} and since the expectation is a linear function, we have
 	
 	\begin{equation*}
 		\mathbb{E}(\|\widetilde{H}_{\beta}\widetilde{{\mathbf{x}}}^{k+1,j}-\widetilde{\boldsymbol{\chi}}^k \|) \leq (1- \frac{(2-\omega)\omega \lambda_1(\widetilde{H}_{\beta})}{(1+\omega \lambda_1(\widetilde{H}_{\beta}))^2 {k}_2(\widetilde{H}_{\beta})})^{j/2} \sqrt{k_2({\widetilde{H}_{\beta}})}
 		\mathbb{E}(\|\widetilde{H}_{\beta}\widetilde{{\mathbf{x}}}^{k+1,0}-\widetilde{\boldsymbol{\chi}}^k\|)
 		\end{equation*}
 and hence, using \eqref{eq:SPDmatrix norm2},
\begin{equation*}
\mathbb{E}(\| {H}_{\beta} {{\mathbf{x}}}^{k+1,j}- {\boldsymbol{\chi}}^k \|) \leq (1- \frac{(2-\omega)\omega \lambda_1(\widetilde{H}_{\beta})}{(1+\omega \lambda_1(\widetilde{H}_{\beta}))^2 {k}_2(\widetilde{H}_{\beta})})^j \sqrt{k_2({\widetilde{H}_{\beta}})k_2(D^{-1})}
\mathbb{E}(\| {H}_{\beta}{{\mathbf{x}}}^{k}- {\boldsymbol{\chi}}^k\|),
\end{equation*}
where we defined ${{\mathbf{x}}}^{k+1,j}:=D^{-1/2}\widetilde{{\mathbf{x}}}^{k+1,j}$ for $j \geq 1$. 
If in the above equation we use the definition of ${{\mathbf{r}}}^{k+1,j}$,  we have 
\begin{equation*}
\mathbb{E}(\|{{\mathbf{r}}}^{k+1,j}\| ) \leq (1- \frac{(2-\omega)\omega \lambda_1(\widetilde{H}_{\beta})}{(1+\omega \lambda_1(\widetilde{H}_{\beta}))^2 {k}_2(\widetilde{H}_{\beta})})^{j/2} \sqrt{k_2({\widetilde{H}_{\beta}})k_2(D^{-1})}
\mathbb{E}(\| {H}_{\beta} {{\mathbf{x}}}^{k}- {\boldsymbol{\chi}}^k\|^2),
\end{equation*}
%Moreover, since,  ${(\mathbb{E}(\|\mathbf{r}^{k,j}\|)^2} \leq {\mathbb{E}(\|\mathbf{r}^{k,j}\|^2)}$ (Jensen's Inequality, see \cite[Th. 34.4]{MR2893652}), we have
%\begin{equation*}
%\mathbb{E}(\|{{\mathbf{r}}}^{k+1,j}\|)^2 \leq (1- \frac{(2-\omega)\omega \lambda_1(\widetilde{H}_{\beta})}{(1+\omega \lambda_1(\widetilde{H}_{\beta}))^2 {k}_2(\widetilde{H}_{\beta})})^j k_2({\widetilde{H}_{\beta}})k_2(D^{-1})
%\mathbb{E}(\| {H}_{\beta} {{\mathbf{x}}}^{k}- {\boldsymbol{\chi}}^k\|)^2,
%\end{equation*}
%and hence, using the tower property (see \cite[Th. 34.4]{MR2893652}),
%\begin{equation*}
%\mathbb{E}(\|{{\mathbf{r}}}^{k+1,j}\|) \leq (1- \frac{(2-\omega)\omega \lambda_1(\widetilde{H}_{\beta})}{(1+\omega \lambda_1(\widetilde{H}_{\beta}))^2 {k}_2(\widetilde{H}_{\beta})})^{j/2} \sqrt{k_2({\widetilde{H}_{\beta}})k_2(D^{-1})}
%\mathbb{E}(\| {H}_{\beta} {{\mathbf{x}}}^{k}- {\boldsymbol{\chi}}^k\|).
%\end{equation*}
and hence, defining 
\begin{equation*}
j^k := \lceil \log_{(1- \frac{(2-\omega)\omega \lambda_1(\widetilde{H}_{\beta})}{(1+\omega \lambda_1(\widetilde{H}_{\beta}))^2 {k}_2(\widetilde{H}_{\beta})})} \frac{2\eta^k}{\sqrt{k_2(\widetilde{H}_{\beta})k_2(D^{-1})} \mathbb{E}({\|H_{\beta}  {\mathbf{x}}^{k}-\boldsymbol{\chi}^k\|})} \rceil, 
\end{equation*}
it holds ${\mathbb{E}(\|\mathbf{r}^{k,j^k}\|)} \leq\eta^k$. 
Reasoning now as in Theorem~\ref{theo:cgNONincreasingIts}, we have 
\begin{equation*}
\mathbb{E}(\|H_{\beta}\mathbf{x}^{k}-\boldsymbol{\chi}^k\|)=\mathbb{E}(\|\mathbf{r}^{k-1}+\beta A^T(A\mathbf{x}^k-\mathbf{b}))\|\leq \mathbb{E}(\|\mathbf{r}^{k-1}\|)+\|A^T\|\mathbb{E}( \|\beta(A\mathbf{x}^k-\mathbf{b})\|),
\end{equation*}
and hence using the hypothesis $\eta^k=R^{k+1}$ and equation \eqref{eq:limit_finitness}, we are able to state the existence of a constant $C>0$ such that
\begin{equation*}
{\frac{2  R^{k+1}}{{\sqrt{k_2(\widetilde{H}_{\beta})k_2(D^{-1})}} \mathbb{E}(\|H_{\beta}\mathbf{x}^{k}-\boldsymbol{\chi}^k\|)}} \geq C \hbox{ for all } k.
\end{equation*}
 We obtain
\begin{equation}\label{eq:SORiterEstimate}
\overline{j}:=\lceil \log_{(1- \frac{(2-\omega)\omega \lambda_1(\widetilde{H}_{\beta})}{(1+\omega \lambda_1(\widetilde{H}_{\beta}))^2 {k}_2(\widetilde{H}_{\beta})})}  C \rceil \geq j^k \hbox{ for all } k.
\end{equation}
From \eqref{eq:SORiterEstimate}, we obtain the first part of the statement observing that $j^k \geq \overline{j}^k$ for all $k$.
%\begin{equation*}
%\mathbb{E}(j_k) \geq \log_{{(1- \frac{(2-\omega)\omega \lambda_1(\widetilde{H}_{\beta})}{(1+\omega \lambda_1(\widetilde{H}_{\beta}))^2 {k}_2(\widetilde{H}_{\beta})})}} C
%\end{equation*}
%iterations are sufficient to guarantee ${\mathbb{E}(\|\mathbf{r}^{k,j}\|)}\leq \eta^k$ for all $k \geq k_0$. 
The last part of the statement follows observing that with this choice of $\eta^k$ the hypotheses of Lemma~\ref{lem:epsol_iALM} are satisfied.
\end{proof}

In the upper panels of  Figure~\ref{fig:inexact_RSGS} we report the quantities analysed in the proof of Lemma~\ref{lem:epsol_iALM} (the legend is consistent with the notation used in Lemma~\ref{lem:epsol_iALM} { except for the fact that we report $\overline{C} \equiv \overline{C}/M$ }). The expectations $\mathbb{E}(\|A\mathbf{x}^k-\mathbf{b}\|)$,  $\mathbb{E}(\|H\mathbf{x}^k + \mathbf{g} -A^T \boldsymbol{\mu}^k\|)$ and $\mathbb{E}(\|\mathbf{d}^k\|)$
are approximated using the empirical mean over $15$ iALM simulations, whereas, for each fixed $k$ and $j$, $\mathbb{E}(\|\mathbf{r}^{k,j}\|)$ is approximated using the empirical mean of $\mathbb{E}(\mathbb{E}(\|\mathbf{r}^{k,j}\||\begin{bmatrix}
\mathbf{x}^k \\ \boldsymbol{\mu}^k
\end{bmatrix}))$  over $15$ trajectories for $[\mathbf{x}^k, \boldsymbol{\mu}^k]^T$ and $15$ simulations of the RSGS step. In the lower panels, we report, for each iALM step and for each simulation, the box-plots of the obtained $\overline{j}^k$ (see equation~\eqref{eq:min_RRSOR}). As Theorem~\ref{theo:sorNONincreasingIts} states and Figure~\ref{fig:inexact_RSGS} confirms, $\overline{j}^k$ shows a \textit{bounded-from-above} behaviour for all the iALM iterations (the choice of the parameters $\beta$ and $R$ is the same as that in Figure~\ref{fig:inexact_CG}).

\begin{figure}[ht!]
	\centering
	\includegraphics[width=\textwidth]{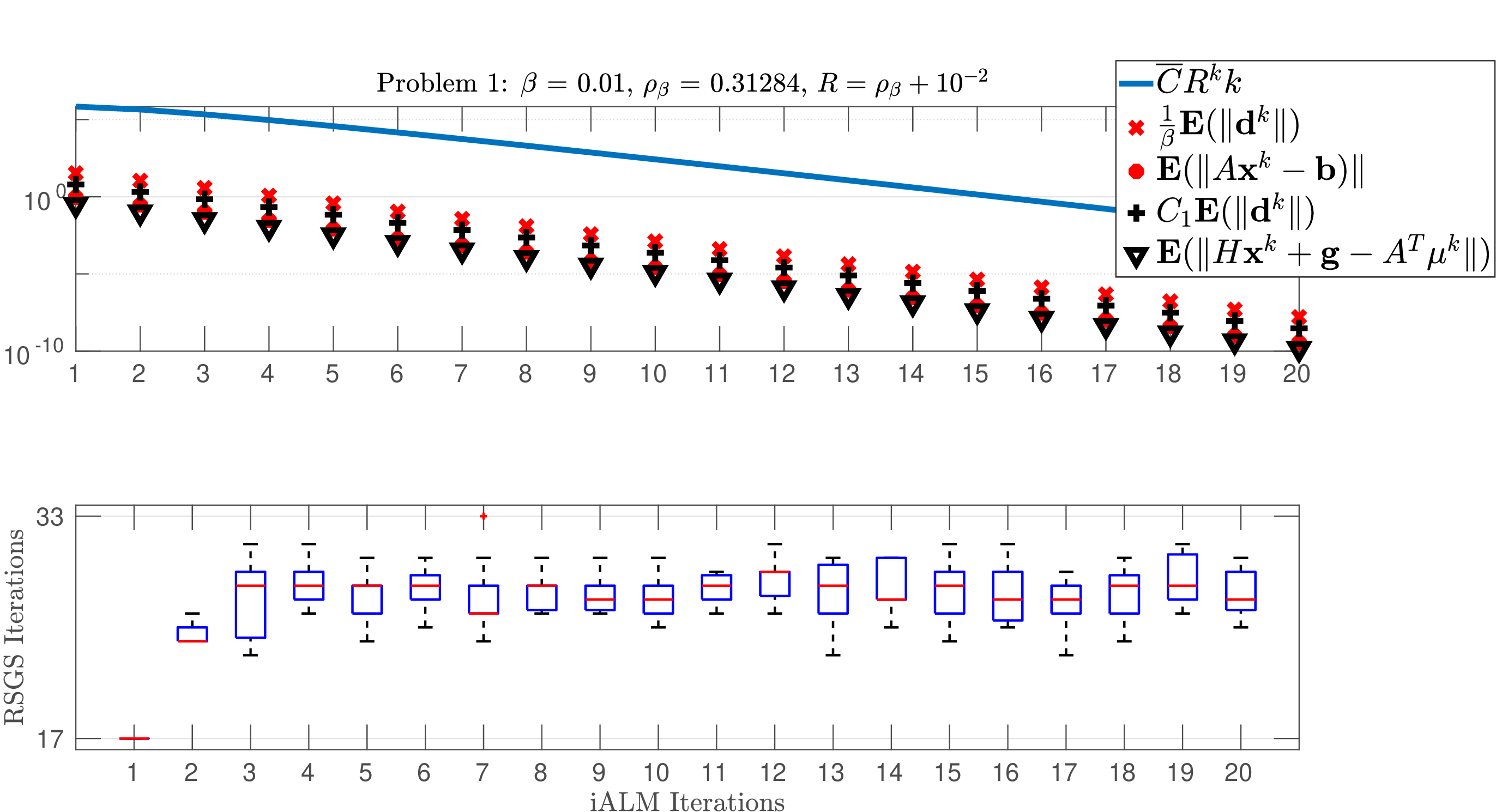}\\
\vspace{0.3cm}
	\includegraphics[width=\textwidth]{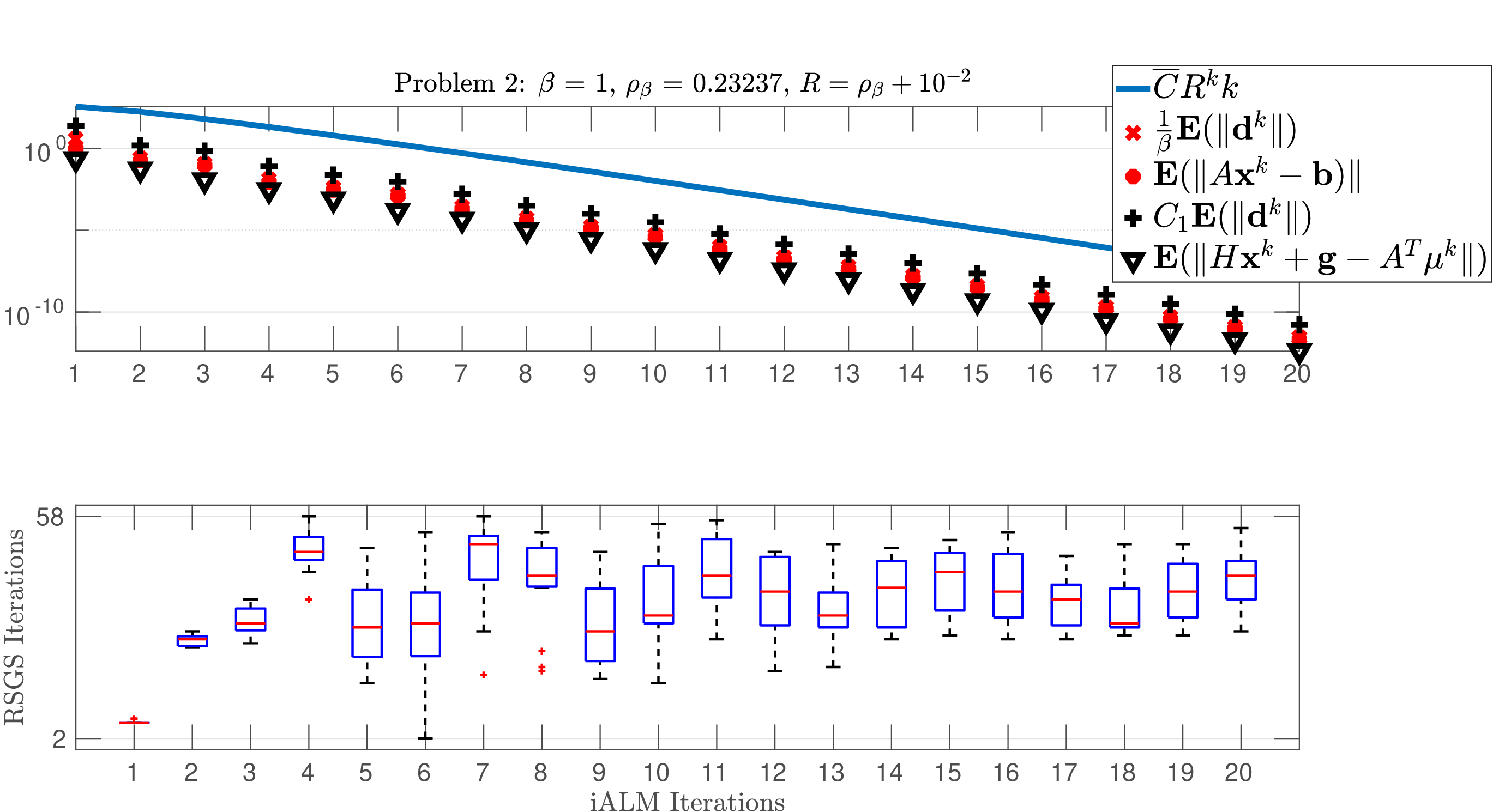}
	\caption{Upper panels: behaviour of the quantities analysed in Lemma~\ref{lem:epsol_iALM} (logarithmic scale on $y$-axis) approximated using the empirical mean over $15$ simulations of iALM. Lower panels: box-plots of the  $\overline{j}^k$'s (see equation~\eqref{eq:min_RRSOR}) obtained in each simulation of iALM when RSSOR is used for the solution of \eqref{eq:inexat_min} using $\{\eta^k\}_k$ and $\{\mathbf{x}^{k+1, 0}\}_k$ as in Theorem~\ref{theo:sorNONincreasingIts}. } \label{fig:inexact_RSGS}
\end{figure}

\section{Interpreting (Random)ADMM as an iALM} \label{sec:ADMM_iALM}
Given a block partition of $\mathbf{x}$, i.e., $\mathbf{x}=[\mathbf{x}_{d_1}, \dots, \mathbf{x}_{d_n}]^T$ with $d_1+\dots+d_n=d$, the $n$-block ADMM (see \cite{MR3439797} and references therein) is defined as

\begin{equation} \label{eq:ADMM_general}
\left\{
\begin{array}{rc}
&\mathbf{x}_{d_1}^{k+1}:= \arg \min_{\mathbf{x}_{d_1} \in \mathbf{R}^{d_1}} \mathcal{L}_\beta([\mathbf{x}_{d_1}, \mathbf{x}^k_{d_2},\dots, \mathbf{x}^k_{d_n}]^T,\boldsymbol{\mu}^k),   \\
&\vdots    \\
&\mathbf{x}_{d_n}^{k+1}:= \arg \min_{\mathbf{x}_{d_n} \in \mathbf{R}^{d_n}} \mathcal{L}_\beta([\mathbf{x}_{d_1}^{k+1}, \mathbf{x}_{d_2}^{k+1},\dots, \mathbf{x}_{d_n}]^T,\boldsymbol{\mu}^k),  \\
\vspace{0.1cm} \\
& \boldsymbol{\mu}^{k+1}:=\boldsymbol{\mu}^k-\beta(A\mathbf{x}^{k+1}-\mathbf{b}).
\end{array}
\right.
\end{equation}

If we apply the iterative method in \eqref{eq:ADMM_general} to solve problem \eqref{eq:QP_problem_ADMM}, splitting $H_{\beta}$ as $H_{\beta}=D-L-L^T$, it is possible to re-write \eqref{eq:ADMM_general} in compact form (see \cite{MR3904457, MR4066996}):
{\footnotesize
	\begin{equation} \label{ADMM_quadrtatic1}
	\begin{bmatrix}
	\mathbf{x}^{k+1} \\
	\boldsymbol{\mu}^{k+1}
	\end{bmatrix}=\underbrace {\begin{bmatrix}
	D-L & 0 \\
	\beta A & I
	\end{bmatrix}^{-1}\begin{bmatrix}
	L^T & A^T \\
	0 & I
	\end{bmatrix}}_{=:G^{ADMM}}\begin{bmatrix}
	\mathbf{x}^{k} \\
	\boldsymbol{\mu}^{k}
	\end{bmatrix}+
	\begin{bmatrix}
	D-L & 0 \\
	\beta A & I
	\end{bmatrix}^{-1}
	\begin{bmatrix}
	\beta A^T \mathbf{b}-\mathbf{g} \\  
	\beta \mathbf{b}
	\end{bmatrix}
	\end{equation}}
Since equation \eqref{ADMM_quadrtatic1} can be written alternatively as

\begin{equation} \label{ADMM_quadratic_implementation1}
\left\{
\begin{array}{rc}
& \mathbf{x}^{k+1}= (D-L)^{-1}L^T\mathbf{x}^{k} +(D-L)^{-1}(A^T\boldsymbol{\mu}^{k}+\beta A^T\mathbf{b}-\mathbf{g}) \\ 
& \boldsymbol{\mu}^{k+1}:=\boldsymbol{\mu}^k-\beta(A\mathbf{x}^{k+1}-\mathbf{b}),
\end{array}
\right.
\end{equation}
we can observe that the first equation in \eqref{ADMM_quadratic_implementation1} is precisely one step of the SOR method with $\omega=1$ (see equation \eqref{eq:SOR_NN}), i.e., ADMM performs exactly one GS iteration for the solution of the linear system $H_{\beta}\mathbf{x}=A^T\boldsymbol{\mu}^{k}+\beta A^T\mathbf{b}-\mathbf{g}$. Let us point out that in \cite{MR3439797} it has been proved that the $n$-block extension of ADMM is not always convergent since there exist examples where the spectral radius of $G^{ADMM}$ in equation  \eqref{ADMM_quadrtatic1} satisfies $\rho(G^{ADMM})>1$. The analysis performed in Sections~\ref{sec:iALM} and \ref{sec:SLS} reveals a simple strategy to remedy this: performing \textit{more} steps of the GS iteration to fulfil the requirements needed on the residuals  will ensure convergence. Indeed,  as proved in Section~\ref{sec:SLS} (deterministic case), 
a constant number of iterations of SOR  per iALM-step is sufficient to guarantee that the produced residuals satisfy the sufficient conditions for convergence. To further underpin the previous claim, in  Figure~\ref{fig:exp4}, we report the behaviour of 
$\|\mathbf{d}^k\|$, $\|A\mathbf{x}^k - \mathbf{b}\|$ and $\|H\mathbf{x}^k+\mathbf{g} -A^T\boldsymbol{\mu}^k\|$ for ADMM and for iALM\&GS where, at each inner iteration,  $10$ GS sweeps are performed. For the particular case of Problem 2 when $\beta=1$ and all the blocks have size one, we have $\rho(G^{ADMM})=1.0148>1$ and the ADMM is not convergent (see the upper panel in Figure~\ref{fig:exp4}). On the contrary, performing more than one GS sweep (lower panel of Figure~\ref{fig:exp4}) is enough to observe a convergent behaviour of all residuals. 

\begin{figure}[ht!]
	\centering
	\includegraphics[width=\textwidth]{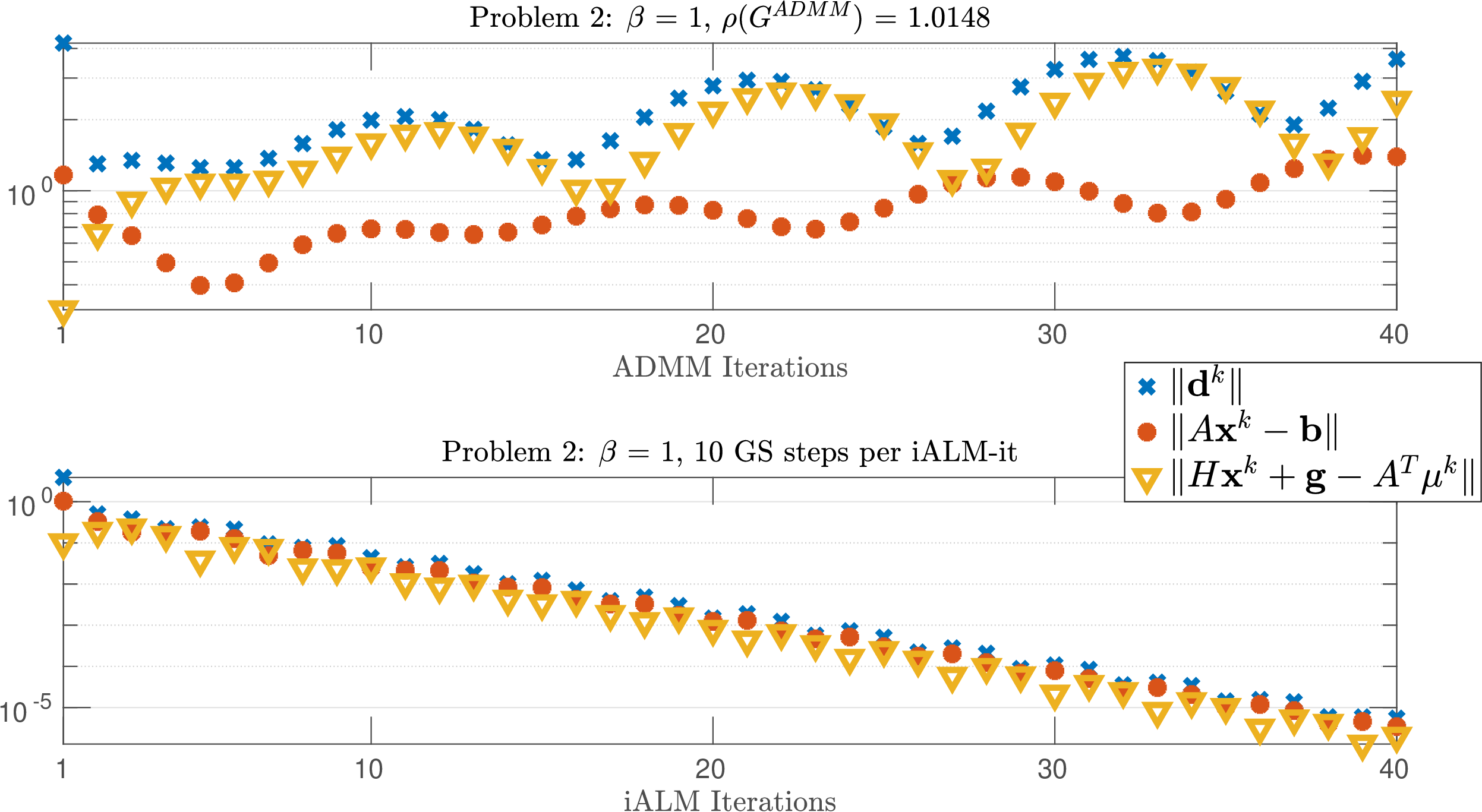}
	\caption{ADMM vs iALM\&GS for Problem 2 (logarithmic scale on $y$-axis). } \label{fig:exp4}
\end{figure}

Exactly the same observation can be made for the RADMM \cite{MR3904457,MR4066996}: this method is obtained considering a block permutation matrix $P^k$ which selects the order for solving the block-equations and then  splitting the matrix $P^kH_{\beta}{P^k}^T$ as  
\begin{equation} \label{eq:RADMM_splitting}
	P^kH_{\beta}{P^k}^T=D_{P^k}-L_{P^k}-L_{P^k}^T
\end{equation}
(the random permutation matrix is selected independently 
from the iterate $\mathbf{x}^k$ and uniformly at random among all possible block-permutation matrices). In more details, if we consider the iterative method

\begin{equation} \label{eq:RandomADMM_general}
\left\{
\begin{array}{rc}
& \hbox{select a permutation } \sigma \hbox{ of } \{1,\dots,n\}  \hbox{ uniformly at random independently from } \mathbf{x}^k,\\
&\mathbf{x}_{d_{\sigma(1)}}^{k+1}:= \arg \min_{\mathbf{x}_{d_{\sigma(1)}} \in \mathbf{R}^{d_{\sigma(1)}}} \mathcal{L}_\beta([\mathbf{x}_{d_{\sigma(1)}}, \mathbf{x}^k_{d_{\sigma(2)}},\dots, \mathbf{x}^k_{d_{\sigma(n)}}]^T,\boldsymbol{\mu}^k),   \\
&\vdots    \\
&\mathbf{x}_{d_{\sigma(n)}}^{k+1}:= \arg \min_{\mathbf{x}_{d_{\sigma(n)}} \in \mathbf{R}^{d_{\sigma(n)}}} \mathcal{L}_\beta([\mathbf{x}_{d_{\sigma(1)}}^{k+1}, \mathbf{x}_{d_{\sigma(2)}}^{k+1},\dots, \mathbf{x}_{d_{\sigma(n)}}]^T,\boldsymbol{\mu}^k),  \\
\vspace{0.1cm} \\
& \boldsymbol{\mu}^{k+1}:=\boldsymbol{\mu}^k-\beta(A\mathbf{x}^{k+1}-\mathbf{b})
\end{array}
\right.
\end{equation}
to solve problem \eqref{eq:QP_problem_ADMM}, using the splitting \eqref{eq:RADMM_splitting}, we can write \eqref{eq:RandomADMM_general} in the fixed point form 

{\footnotesize
	\begin{equation} \label{eq:ADMM_FIXEDPOINT}
	\begin{bmatrix}
	\mathbf{x}^{k+1} \\ \boldsymbol{\mu}^{k+1}
	\end{bmatrix}  \! \!=  \! \!\underbrace{ \begin{bmatrix}
	P_k^T & 0 \\
	0 & I
	\end{bmatrix} \! \! \! \begin{bmatrix}
	D_{P_k}  \! \!- \!L_{P_k} & 0 \\
	\beta AP_k^T & I
	\end{bmatrix}^{-1} \! \! 
	\begin{bmatrix}
	L_{P_k}^TP_k & P_kA^T \\
	0 & I
	\end{bmatrix} }_{=:G_{\beta}^{P_k}} \!  \!
	\begin{bmatrix}
	\mathbf{x}^k \\ \boldsymbol{\mu}^{k}
	\end{bmatrix} \!
	+ \begin{bmatrix}
	P_k^T & 0 \\
	0 & I
	\end{bmatrix}  \! \!\!\begin{bmatrix}
	D_{P_k}  \!\!- \!L_{P_k} & 0 \\
	\beta AP_k^T & I
	\end{bmatrix}^{-1} \! \!\begin{bmatrix}
	P_k(\beta A^T \mathbf{b}-\mathbf{g}) \\  \beta \mathbf{b}
	\end{bmatrix},
	\end{equation}
}
and hence
\begin{equation} \label{ADMM_quadratic_implementation2}
\left\{ \hspace*{-0.5cm}%
\begin{array}{rc} 
  & \mathbf{x}^{k+1} \!=  \!{P^k}^T[(D_{P^k} \!- \! L_{P^k})^{-1} L_{P^k}^T] P^k\mathbf{x}^{k}  \!+ \!{P^k}^T(D_{P^k} \!- \!L_{P^k})^{-1}P^k(A^T\boldsymbol{\mu}^{k} \!+ \!\beta A^T\mathbf{b} \!- \!\mathbf{g}) \\
  & \boldsymbol{\mu}^{k+1}:=\boldsymbol{\mu}^k-\beta(A\mathbf{x}^{k+1}-\mathbf{b}).
\end{array}
\right.
\end{equation}
The first equation in \eqref{ADMM_quadratic_implementation2} coincides exactly with one iteration of the RSSOR with $\omega=1$ (see equation \eqref{eq:shuflled_sor}) for the solution of the linear system $H_{\beta}\mathbf{x}=A^T\boldsymbol{\mu}^{k}+\beta A^T\mathbf{b}-\mathbf{g}$. On the other hand,  as proved in Theorem~\ref{theo:sorNONincreasingIts}, the number of RSSOR sweeps per iALM-step sufficient to obtain an {expected} residual which ensures the a.s. convergence, is uniformly bounded above by a constant. 
We find that this is a noteworthy improvement of the results obtained in \cite{MR3904457,MR4066996,mihic2019managing}. Indeed, in these works,  only the \textit{the convergence in expectation} of the iterates produced by \eqref{ADMM_quadratic_implementation2} has been proved, i.e., the convergence to zero of $\|\mathbb{E}(\begin{bmatrix}
\mathbf{x}^k \\
\boldsymbol{\mu}^k
\end{bmatrix})-\begin{bmatrix}
\overline{\mathbf{x}} \\
\overline{\boldsymbol{\mu}}
\end{bmatrix}\|$. To be precise, {using the notation introduced in \eqref{eq:ADMM_FIXEDPOINT}}, the authors prove that  $\overline{\rho}_{\beta}:=\rho({\overline{G}_{\beta}})<1$ where
\begin{equation*}
	\overline{G}_{\beta}:=\mathbb{E}(G_{\beta}^{P})=\frac{1}{|\mathcal{P}|}\sum_{P \in \mathcal{P}} G_{\beta}^{P}
\end{equation*}
and $\mathcal{P}$ is a specific subset of all permutation matrices ($\mathcal{P}$ is the subset of block permutation matrices with blocks of order $n$ in \cite{MR3904457,MR4066996} and, in \cite{mihic2019managing}, $\mathcal{P}$ is the subset of the permutation matrices obtained as $P=P_1P_2$, where $P_1$ is a block permutation matrix with blocks of order $n$ and $P_2$ is a permutation corresponding to a partition of $d$ elements into $n$ groups). 

Overall, as already pointed out in \cite[Sec. 2.2.4]{mihic2019managing}, the convergence in expectation may not be a good indicator of the robustness and the effectiveness of RADMM as there may exist problems characterized by a high $\|\mathbb{V}(G_{\beta}^{P})\|$: we find that switching from a convergence in expectation to an a.s. convergence with provable expected worst case complexity as stated in Theorem~\ref{theo:sorNONincreasingIts}, could be beneficial for the solution of such problems.

Even in this case, to further underpin the previous claim, we report in Figure~\ref{fig:exp4bis} the behaviour of $\|\mathbf{d}^k\|$, $\|A\mathbf{x}^k - \mathbf{b}\|$ and $\|H\mathbf{x}^k+\mathbf{g} -A^T\boldsymbol{\mu}^k\|$ for RADMM and for iALM\&RSGS where, at each inner iteration, $10$ RSGS sweeps are performed. As it is clear from the comparison between the upper panels of Figures~\ref{fig:exp4} and  \ref{fig:exp4bis} (and expected from the results obtained in \cite{MR4066996,MR3904457}), the introduction of a randomization procedure in the ADMM scheme is able to \textit{mitigate} the divergence in the case of Problem 2. At the same time, analogously of what was observed in Figure~\ref{fig:exp4} for the deterministic case, the benefits of performing more than one RSGS sweep per iALM-step are evident (lower panel of Figure~\ref{fig:exp4bis}).   
\begin{figure}[ht!]
	\centering
	\includegraphics[width=\textwidth]{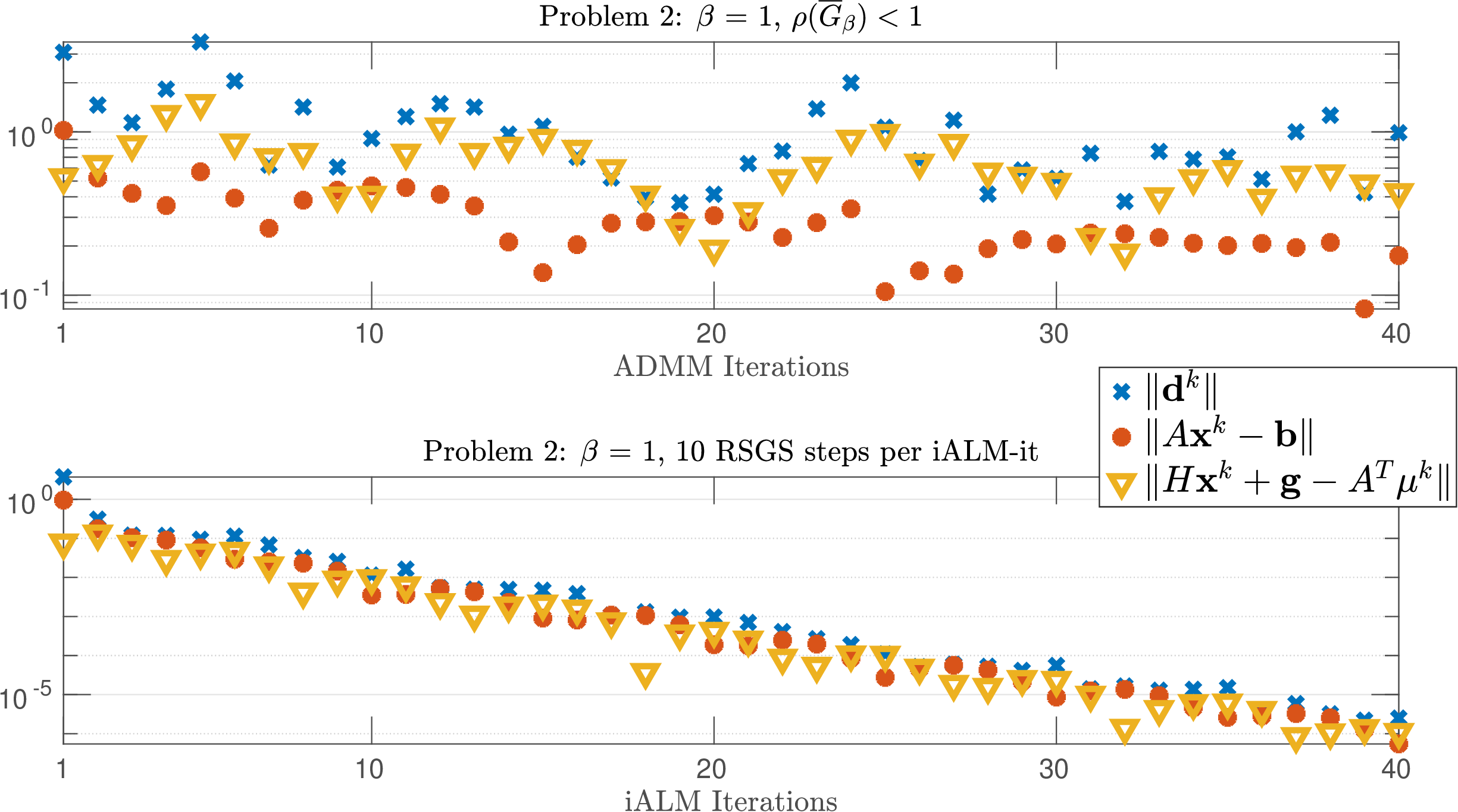}
	\caption{Random ADMM vs iALM\&RSGS for Problem 2 (logarithmic scale on $y$-axis). } \label{fig:exp4bis}
\end{figure} 

\section{Conclusions}
In this work we studied the inexact Augmented Lagrangian Method (iALM) for the solution of problem \eqref{eq:QP_problem_ADMM}. 
Using a splitting operator perspective,
we proved that if the  \textit{amount of introduced inexactness} (which could be modelled also with a random variable) decreases (in expectation) accordingly to suitably chosen $R^{k}$ where $R<1$, then we are able to give explicit asymptotic rate of convergence of the iALM (see Lemma~\ref{lem:epsol_iALM}).
Moreover,  even if the above mentioned condition requires an increasing accuracy in the linear systems to be solved at each iteration, we proved that when these linear systems are solved using the Conjugate  Gradient (CG) method or the Successive-Over-Relaxation method (SOR) and its Randomly Shuffled version (RSSOR), the number of iterations sufficient to satisfy the convergence requirements can be uniformly bounded from above (see Section~\ref{sec:SLS}). Finally, using the developed theory and interpreting the $n$-block (Random)Alternating Direction Method of Multipliers ((R)ADMM) as an iALM which performs exactly one (RS)SOR sweep  to obtain the approximate solutions of the inner linear systems, we provided computational evidence which demonstrates that the very well known convergence issues of the $n$-block (R)ADMM could be remedied if more than one (RS)SOR sweep for every iALM iteration were permitted. 

\section*{Acknowledgements}
The authors are in debt with M. Rossi (University of Milano-Bicocca) for the fruitful discussions on some technical details about the probabilistic case.  

\printbibliography		
\end{document}